\documentclass[11pt,reqno]{amsart}
\usepackage{mathrsfs}
\usepackage{url}
\usepackage{mathtools}
\usepackage{latexsym,epsfig,amssymb,amsmath,amsthm,color,url,bm}
\usepackage[inline,shortlabels]{enumitem}
\usepackage{hyperref}
\usepackage[foot]{amsaddr}
\usepackage{amsmath,amsbsy}
\RequirePackage[numbers]{natbib}

\allowdisplaybreaks 
 \numberwithin{equation}{section}
\newtheorem{theorem}{Theorem}[section]

\newtheorem{lemma}[theorem]{Lemma}

\newcommand\Item[1][]{%
  \ifx\relax#1\relax  \item \else \item[#1] \fi
  \abovedisplayskip=0pt\abovedisplayshortskip=0pt~\vspace*{-\baselineskip}}

\theoremstyle{definition}
\newtheorem{defn}[theorem]{Definition}

\theoremstyle{definition}
\newtheorem{remark}[theorem]{Remark}

\DeclareMathOperator{\Prob}{\mathbf{P}}
\DeclareMathOperator{\E}{\mathbf{E}}
\DeclareMathOperator{\bin}{Bin}
\DeclareMathOperator{\poi}{Poisson}
\DeclareMathOperator{\NW}{NW}
\DeclareMathOperator{\NL}{NL}
\DeclareMathOperator{\ND}{ND}
\DeclareMathOperator{\nw}{nw}
\DeclareMathOperator{\nl}{n\ell}
\DeclareMathOperator{\FP}{FP}
\DeclareMathOperator{\bnw}{\mathbf{n}\mathbf{w}}
\DeclareMathOperator{\bnl}{\mathbf{n}\ensuremath{\boldsymbol\ell}}
\DeclareMathOperator{\MW}{MW}
\DeclareMathOperator{\ML}{ML}
\DeclareMathOperator{\MD}{MD}
\DeclareMathOperator{\mw}{mw}
\DeclareMathOperator{\ml}{m\ell}
\DeclareMathOperator{\bmw}{\mathbf{m}\mathbf{w}}
\DeclareMathOperator{\bml}{\mathbf{m}\ensuremath{\boldsymbol\ell}}
\DeclareMathOperator{\ESW}{ESW}
\DeclareMathOperator{\ESL}{ESL}
\DeclareMathOperator{\EEW}{EEW}
\DeclareMathOperator{\EEL}{EEL}
\DeclareMathOperator{\esw}{esw}
\DeclareMathOperator{\esl}{es\ell}
\DeclareMathOperator{\eew}{eew}
\DeclareMathOperator{\eel}{ee\ell}
\DeclareMathOperator{\besw}{\mathbf{e}\mathbf{s}\mathbf{w}}
\DeclareMathOperator{\beel}{\mathbf{e}\mathbf{e}\ensuremath{\boldsymbol\ell}}
\DeclareMathOperator{\nd}{nd}
\DeclareMathOperator{\md}{md}
\DeclareMathOperator{\F}{\mathcal{F}}
\DeclareMathOperator{\tv}{TV}
\DeclareMathOperator{\blbl}{bb}
\DeclareMathOperator{\br}{br}
\DeclareMathOperator{\rr}{rr}
\DeclareMathOperator{\out}{Out}

\title{Combinatorial games on multi-type Galton-Watson trees}
\date{}
\author{Moumanti Podder}
\address{Moumanti Podder, Indian Institute of Science Education and Research (IISER) Pune, Dr.\ Homi Bhabha Road, Pashan, Pune 411008, Maharashtra, India.}
\email{moumanti@iiserpune.ac.in}

\begin{document}
\bibliographystyle{plainnat}

\begin{abstract}
When normal and mis\`{e}re games are played on bi-type binary Galton-Watson trees (with vertices coloured blue or red and each having either no child or precisely $2$ children), with one player allowed to move along monochromatic edges and the other along non-monochromatic edges, the draw probabilities equal $0$ unless every vertex gives birth to one blue and one red child. On bi-type Poisson trees where each vertex gives birth to $\poi(\lambda)$ offspring in total, the draw probabilities approach $1$ as $\lambda \rightarrow \infty$. We study such \emph{novel} versions of normal, mis\`{e}re and escape games on rooted multi-type Galton-Watson trees, with the ``permissible" edges for one player being disjoint from those of her opponent. The probabilities of the games' outcomes are analyzed, compared with each other, and their behaviours as functions of the underlying law explored.
  
\end{abstract}



\maketitle

\section{Introduction}\label{sec:intro}
\begin{figure}[h!]
  \centering
    \includegraphics[width=0.8\textwidth]{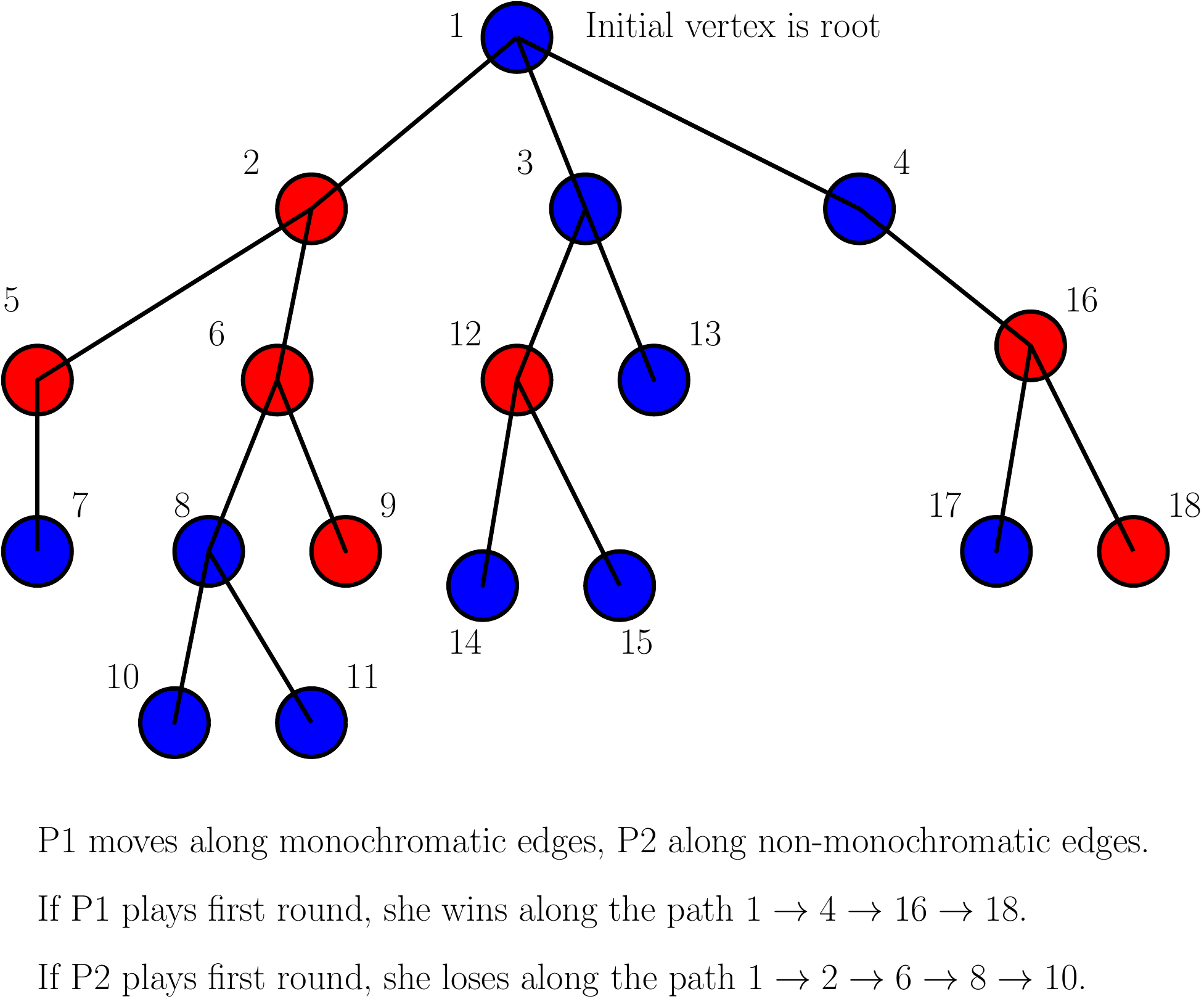}
  \caption{Normal game on a bi-type tree}
  \label{normal_illus}
\end{figure}

\sloppy This paper is dedicated to the analysis of three well-known games -- the \emph{normal}, the \emph{mis\`{e}re} and the \emph{escape} games -- on \emph{rooted multi-type Galton-Watson trees}. These games are played on directed, acyclic graphs. Given a realization of a rooted random tree, we assign the following notion of direction to each of its edges: edge $\{u, v\}$ is directed from $u$ to $v$ if $u$ is the parent of $v$. Each of these games involves two players and a token. The vertex on which the token is placed at the beginning of the game is known as the \emph{initial vertex}. The players take turns to move the token along the directed edges, conforming to the rules that we describe in detail in Definition~\ref{defn:games}. 

The (extremely broad class of) \emph{combinatorial games} (see, for example, \cite{survey_games, complexity_appeal} for a general introduction to these games as well as a discussion of the vast literature devoted to this topic) are two-player games with perfect information, no chance moves, and the possible outcomes being victory for one player (and loss for the other) and draw for both players. Countless intriguing and natural mathematical problems that belong to complexity classes harder than \emph{NP} constitute two-player combinatorial games. Besides, these games have applications / connections to disciplines such as mathematical logic, automata theory, complexity theory, graph and matroid theory, networks, error-correcting codes, online algorithms, and even, outside of mathematics, to biology, psychology, economics, insurance, actuarial studies and political sciences. 

\subsection{The (simple) Galton-Watson branching process} A \emph{rooted Galton-Watson} (GW) branching process $\mathcal{T}_{\chi}$, introduced in \cite{galton_watson} (independently studied in \cite{bienayme}) as a model to investigate the extinction of ancestral family names, begins with the root $\phi$ giving birth to a random number $X$ of children where $X$ follows the \emph{offspring distribution} $\chi$ (a probability distribution supported on $\mathbb{N}_{0}$). If $X = 0$, we stop the process. If $X = k \in \mathbb{N}$, the children of $\phi$ are named $v_{1}, \ldots, v_{k}$ in some order, and $v_{i}$ gives birth to $X_{i}$ children with $X_{1}, \ldots, X_{k}$ i.i.d.\ $\chi$. This process continues, and it survives (i.e.\ continues forever) with positive probability iff the expectation of $\chi$ exceeds $1$. We refer the reader to \cite{athreya_vidyashankar}, \cite{athreya_jagers} and \cite{athreya_ney} for further reading on GW trees. 

\subsection{The games when played on a simple GW tree}\label{subsec:simple_version} We call the players \emph{P1} and \emph{P2} when they play the normal and mis\`{e}re games, and \emph{Stopper} and \emph{Escaper} when they play the escape game. A realization $T$ of $\mathcal{T}_{\chi}$ is fixed, the token placed on an initial vertex $v$ of $T$, the players take turns to move the token along directed edges, and the outcomes of the games are decided as follows: 
\begin{enumerate}
\item \textbf{Normal game:} Whoever fails to make a move for the first time in the game, loses. If the token never reaches a leaf vertex throughout the game, it results in a draw. See Figure~\ref{normal_illus}.
\item \textbf{Mis\`{e}re game:} Whoever fails to make a move for the first time in the game, wins. If the token never reaches a leaf vertex throughout the game, then the game results in a draw.
\item \textbf{Escape game:} If either player fails to make a move, Stopper wins. Else Escaper wins. This game never results in a draw.
\end{enumerate}
We are concerned with \emph{optimal play}, i.e.\ when the game does not end in a draw, the player destined to win tries to win as quickly as possible, while her opponent tries to prolong it as much as possible. Analysis of these games on rooted simple GW trees has been carried out in \cite{holroyd_martin}. 

\subsection{Motivation for studying such combinatorial games} We dwell here on several combinatorial games that have been studied on random structures, along with their myriad theoretical applications and connections to other areas of mathematics. \cite{percolation_games} studies normal and a variant of mis\`{e}re games on percolation clusters of oriented Euclidean lattices. This \emph{percolation game} assigns one of the labels ``trap", ``target" and ``open" to each site of $\mathbb{Z}^{2}$ with probabilities $p$, $q$ and $1-p-q$ respectively, and the players take turns to move a token from its current position $(x,y)$ to either $(x+1,y)$ or $(x,y+1)$. If a player moves to a target, she wins immediately, and if she moves to a trap, she loses immediately. The game's outcome can be interpreted in terms of the evolution of a one-dimensional discrete-time probablistic cellular automaton (PCA) -- specifically, the game having no chance of ending in a draw is shown to be equivalent to the ergodicity of this PCA. \cite{percolation_games} also establishes a connection between the percolation game with $q = 0$ (called the \emph{trapping game}) on directed graphs in higher dimensions and the hard-core model on related undirected graphs with reduced dimensions. \cite{trapping_games} studies the trapping game on undirected graphs. The players take turns to move the token from the vertex of its current position to an adjacent vertex that has never been visited before. The player unable to make a move loses. The outcome of this game is shown to have close ties with maximum-cardinality matchings, and a draw in this game relates to the sensitivity of such matchings to boundary conditions. \cite{wastlund} studies a related, two-person zero-sum game called \emph{exploration} on a \emph{rooted distance model}, to analyze minimum-weight matchings in edge-weighted graphs. In a related game called \emph{slither} (\cite{slither}), the players take turns to claim yet-unclaimed edges of a simple, undirected graph, such that the chosen edges, at all times, form a path, and whoever fails to move, loses. This too serves as a tool for understanding maximum matchings in graphs. 

The \emph{maker-breaker positional games} (\cite{positional_games_book}) involve a set $X$, a collection $\mathcal{F}$ of subsets of $X$, and $a, b \in \mathbb{N}$. \emph{Maker} and \emph{Breaker} take turns to claim yet-unclaimed elements of $X$, with Maker choosing $a$ elements at a time and Breaker $b$ elements at a time, until all elements of $X$ are exhausted. Maker wins if she claims all elements of a subset in $\mathcal{F}$. When this game is played on a graph, the players take turns to claim yet-unclaimed edges, and Maker wins if the subgraph induced by her claimed edges satisfies a desired property (e.g.\ it is connected, it forms a clique of a given size, a Hamiltonian cycle, a perfect matching or a spanning tree). The game is \emph{unbiased} when $a = b$, and \emph{biased} otherwise. This game has intimate connections with existential fragments of first order and monadic second order logic on graphs. \cite{milos_thesis} and \cite{milos_tibor} study the threshold probability $p_{c}$ beyond which Maker has a winning strategy when this game is played on Erd\H{o}s-R\'{e}nyi random graphs $G(n,p)$; \cite{hamiltonian_maker_breaker} studies the game for Hamiltonian cycles on the complete graph $K_{n}$; \cite{maker_breaker_geometric} studies the game on random geometric graphs; \cite{biased_random_boards} studies the \emph{critical bias} $b^{*}$ of the $(1 : b)$ biased game on $G(n, p(n))$ for $p(n) = \Theta\left(\ln n/n\right)$. In addition, \cite{milos_thesis} studies  the game where Maker wins if she can claim a non-planar graph or a non-$k$-colourable graph. \cite{biased_positional} indicates a deep connection between positional games on complete graphs and the corresponding properties being satisfied by a random graph. This follows from \emph{Erd\H{o}s' probabilistic intuition}, which states that the course of a combinatorial game between two players playing optimally often resembles the evolution of a purely random process. 

Finally, the \emph{Ehrenfeucht-Fra\"{i}ss\'{e} games} play a pivotal role in our understanding of first and monadic second order logic on random rooted trees and random graphs (see, for example, \cite{spencer_threshold, spencer_thoma, spencer_stjohn, kim, bohman, pikhurko, verbitsky, maksim_1, maksim_2, maksim_3, maksim_4, maksim_5, maksim_6, podder_1, podder_2, podder_3, podder_5}).

\subsection{Novelty of the games studied in this paper} To the best of our knowledge, the normal, mis\`{e}re and escape games have not been studied in the premise of graphs with coloured vertices. The rules to be followed when the games are played on rooted multi-type GW trees, as described in Definition~\ref{defn:games}, are significantly different from those studied in \cite{holroyd_martin} (see also \S\ref{subsec:simple_version}). In particular, the directed edges along which one player is allowed to move are no longer the same as those along which the other player is allowed to move, thereby breaking the symmetry and making the analysis far more complicated. This paper thus serves as a pretty broad generalization of the work in \cite{holroyd_martin}. It no longer suffices for the players to take note of the leaf vertices of the tree and the paths that lead down to those vertices -- they are now required to use a \emph{look-ahead strategy} that must take into account the colours of the endpoints of every directed edge they may have to traverse. Moreover, much of the analysis in this paper involves multivariable functions defined on subsets of $[0,1]^{r}$ for some $r \in \mathbb{N}$, and the calculus used to draw conclusions about corresponding functions defined on subsets of $[0,1]$ in \cite{holroyd_martin} no longer applies to our set-up.  

There are \textbf{several sharp points of contrast} between some of the results in this paper and those of \cite{holroyd_martin}. First, we note that despite the far more complicated set-up of this paper compared to that of \cite{holroyd_martin}, the inequalities in Theorem~\ref{thm:main_2} (some of which are analogous to those in [\cite{holroyd_martin}, Theorem 2]) require very few, and intuitively very reasonable, assumptions (see, for example, Remark~\ref{rem_intuitive} in relation to Theorem~\ref{thm:main_2}, part \ref{main_2_part_3}). The inequalities in \ref{main_2_part_3} are, rather surprisingly, \emph{very} different from [\cite{holroyd_martin}, Theorem 2, part (iii)], and those in \ref{main_2_part_2} are, in some sense, stronger than the inequalities in [\cite{holroyd_martin}, Theorem 2, part (ii)]. We take into account the possibility of offspring distributions having infinite expectations in Theorem~\ref{thm:main_3}, which is not the case with [\cite{holroyd_martin}, Theorem 3]. The analysis of the games on bi-type binary GW trees in Theorem~\ref{thm:main_example_1} yields a strikingly different picture from that obtained while studying them on the simple binary GW tree (see [\cite{holroyd_martin}, Proposition 3, part (i)]), specifically in terms of phase transitions for draw probabilities in the normal and mis\`{e}re games and probabilities of Escaper winning the escape game. In some sense, Theorem~\ref{thm:main_example_1} gives us results reminiscent of $0-1$ laws. Theorem~\ref{thm:main_example_2} shows that on bi-type Poisson GW trees, where each vertex, irrespective of its colour, has $\poi(\lambda)$ children in total, the draw probabilities of the normal and mis\`{e}re games and the probabilities of Escaper winning the escape game all approach $1$ as $\lambda \rightarrow \infty$.

\subsection{Inspirations and possible applications of the games studied in this paper} One of the primary inspirations for studying these particular versions of the games lies in our interest in understanding \emph{illegal moves} in mathematical games. As seen in Definition~\ref{defn:mult_GW}, a move made by one of the players along an edge not \emph{permissible} for her is considered illegal and is not allowed. One may think of more ``relaxed" versions of these games where a player is allowed to make illegal moves at most $k$ times in the game, for some $k \in \mathbb{N}$, or the total number of illegal moves made by both players is allowed to be at most $k$. Such games are reminiscent of rules employed by the World Chess Federation (FIDE): when in a timed game, the first illegal move by a player awards a certain number of minutes worth of extra time to her opponent, while a second illegal move forfeits the game. They also resemble \emph{liar} and \emph{half-liar games} where one of the players, Carole, is allowed to lie a certain number of times at the most when answering questions asked by her opponent, Paul.

As far as \textbf{applications} are concerned, we mention here a curious connection -- one that could be exploited for investigating a generalized version of the PCA studied in \cite{percolation_games} -- between the games analyzed in this paper and the percolation games (in particular, the trapping games). Although this paper concerns itself with games played on rooted trees, let us for a moment imagine the games being played on an oriented lattice $\mathbb{Z}^{2}$ with each vertex coloured either blue (denoted $b$) or red (denoted $r$). Any directed edge in this lattice is either of the form $((x,y), (x+1,y))$ or of the form $((x,y), (x,y+1))$. Suppose P1 is allowed to move only along monochromatic directed edges and P2 only along non-monochromatic directed edges. When it comes to the normal game, we may think of each site $u = (x,y)$ being assigned one of the following labels according to the fate P1 meets with if she moves the token to $u$:
\begin{itemize}
\item if $\sigma(u) = b$ with no red child, or $\sigma(u) = r$ with no blue child, $u$ is a target for P1;
\item if $\sigma(u) = b$ and $u$ has at least one red child $v$ with no red child of its own, or if $\sigma(u) = r$ and $u$ has at least one blue child $v$ with no blue child of its own, then $u$ is a trap for P1;
\item in all other cases, $u$ is marked open for P1.
\end{itemize}
Likewise, when it comes to deciding P2's fate upon reaching $u$,
\begin{itemize}
\item if $\sigma(u) = b$ with no blue child, or $\sigma(u) = r$ with no red child, $u$ is a target for P2;
\item if $\sigma(u) = b$ and $u$ has at least one blue child $v$ with no red child of its own, or if $\sigma(u) = r$ and $u$ has at least one red child $v$ with no blue child of its own, then $u$ is a trap for P2.
\end{itemize}
Following the notations used in \cite{percolation_games}, for $i = 1, 2$, we let $\eta_{i}(u) = W$ if the game that begins with the token at $u$ and Pi playing the first round is won by Pi, $\eta_{i}(u) = L$ if it is lost by Pi, and $\eta_{i}(u) = D$ if it ends in a draw. If $u = (x,y)$ is a trap for P2, then $\eta_{1}(u) = W$, and if $u$ is a target for P2, then $\eta_{1}(u) = L$. Setting $\out(u) = \{(x+1,y), (x,y+1)\}$, if $\sigma(u) = b$ and $u$ is open for P2, then 
\begin{itemize}
\item $\eta_{1}(u) = W$ if there exists at least one $v \in \out(u)$ with $\sigma(v) = b$ and $\eta_{2}(v) = L$;
\item $\eta_{1}(u) = L$ if $\sigma(v) = b$ implies that $\eta_{2}(v) = W$ for every $v \in \out(u)$;
\item $\eta_{1}(u) = D$ otherwise.
\end{itemize}
We analogously derive the recursions when $\sigma(u) = r$ and $u$ is open for P2. Although this yields a more complicated set-up than that in \cite{percolation_games}, whether we should consider $\eta_{1}(u)$ or $\eta_{2}(u)$ for any $u = (x,y)$ ought to depend on whether the parities of $x_{0}+y_{0}$ and $x+y$ are the same and on whether P1 or P2 plays the first round, where $(x_{0},y_{0})$ is the site from which the game begins. It seems likely that this version of the percolation games will allow us to study a broader, more general class of probabilistic cellular automata -- one where we may have to consider not one but two states $\eta_{t}^{(1)}(n)$ and $\eta_{t}^{(2)}(n)$ of any site $n \in \mathbb{Z}$ at any point of time $t$, and the alphabet needs to be extended from $\{0,1\}$ to $\{0b, 0r, 1b, 1r\}$. This may also aid us in studying a statistical mechanical model, with hard constraints, that is a generalization of the hard-core model studied in \cite{percolation_games} and that involves the assignment of a label from $\{0,1\}$ \emph{and} a colour from $\{b, r\}$ to each vertex of the graph.

We strongly suspect that the games we study here are also applicable in extending the results connecting trapping games with maximum matchings in \cite{trapping_games}. As an example, let us consider the trapping game on a graph each of whose vertices has been assigned one of two colours: blue and red. P1, as above, is allowed to move the token along monochromatic edges and P2 along non-monochromatic edges, making sure to never re-visit a vertex. Inspired by the conclusion of [\cite{trapping_games}, Proposition 4], we surmise that in order for P1 to win the game starting from a vertex $v$, it is perhaps necessary for $v$ to be contained in every \emph{colour-coordinated maximum matching} of the graph, i.e.\ every maximal matching where each vertex and its partner must be of the same colour.

We allude here to a related, and analytically very similar, version of each of these games that may also be studied. Fixing a non-empty proper subset $S$ of $[m]$, P1 (in the normal and mis\`{e}re games) and Stopper (in the escape game) are allowed to move the token along a directed edge $(u,v)$ as long as $\sigma(v) \in S$, whereas P2 and Escaper must move only along directed edges $(u, v)$ with $\sigma(v) \in [m] \setminus S$. An even more generalized version involves two non-empty subsets $S_{1}$ and $S_{2}$ of $[m]$ with $S_{1} \cup S_{2} \subsetneq [m]$. P1 and Stopper are allowed to move along $(u,v)$ with $\sigma(v) \in [m] \setminus S_{2}$, and P2 and Escaper along $(u,v)$ with $\sigma(v) \in [m] \setminus S_{1}$. Consequently, the sets of permissible edges for the two players now overlap. Careful analysis is required to understand the impact of such generalizations on the outcome probabilities of these games.  

\subsection{Main definitions and illustrative examples}\label{subsec:defns}
Here, we define the rooted multi-type Galton-Watson branching process and the games we study on it, followed by a couple of motivating examples.
\begin{defn}\label{defn:mult_GW}
Given a finite set of colours $[m]$, a probability vector $\mathbf{p} = (p_{1}, \ldots, p_{m})$ and probability distributions $\chi_{1}, \ldots, \chi_{m}$ with each $\chi_{j}$ supported on $\mathbb{N}_{0}^{m}$, the rooted multi-type Galton-Watson branching process $\mathcal{T} = \mathcal{T}_{[m] , \mathbf{p}, \bm{\chi}}$, with $\bm{\chi} = \left(\chi_{j}: j \in [m]\right)$, is generated as follows. The root $\phi$ is assigned a colour $\sigma(\phi)$ from $[m]$ according to $\mathbf{p}$. From there onward, every vertex $v$ of the tree, provided that its colour $\sigma(v)$ equals $i$ for some $i \in [m]$, gives birth, independent of all else, to $X_{v,j}$ many offspring of colour $j$ for all $j \in [m]$, where $(X_{v,1}, \ldots, X_{v,m}) \sim \chi_{i}$, i.e.\ for all $n_{1}, \ldots, n_{m} \in \mathbb{N}_{0}$,
\begin{equation}\label{eq:mult_GW}
\Prob\left[X_{v,j} = n_{j} \text{ for all } j \in [m]\big|\sigma(v) = i\right] = \chi_{i}(n_{1}, \ldots, n_{m}).
\end{equation}
\end{defn}
We refer the reader to [\cite{athreya_ney}, Chapter V, Pages 181-228] and [\cite{karlin_taylor}, Chapter 8, Pages 392-442] for further reading.

\begin{defn}\label{defn:games}
For each $j \in [m]$, we fix a non-empty, proper subset $S_{j}$ of $[m]$. As previously mentioned, every edge $(u,v)$ of the rooted tree is directed from $u$ to $v$ where $u$ is the parent of $v$.
\begin{enumerate}
\item We define a directed edge $(u,v)$ to be permissible for P1 / Stopper if $\sigma(v) \in S_{\sigma(u)}$. 
\item We define a directed edge $(u,v)$ to be permissible for P2 / Escaper if $\sigma(v) \in [m] \setminus S_{\sigma(u)}$.
\end{enumerate}
In each of these games, in every round, the player makes sure to move along a directed edge permissible for her. The outcome of each game is decided via the same rules as outlined in \S\ref{subsec:simple_version}.
\end{defn}

We consider a couple of examples. In the first, we set $S_{j} = \{j\}$ for each $j \in [m]$. This means that P1 and Stopper are allowed to move only along \emph{monochromatic} directed edges, i.e.\ $(u,v)$ with $\sigma(u) = \sigma(v)$, whereas P2 and Escaper are allowed to move only along \emph{non-monochromatic} directed edges, i.e.\ $(u,v)$ with $\sigma(u) \neq \sigma(v)$. Note that swapping these rules for Stopper and Escaper yields a different escape game altogether.

The second example can be thought of as a generalization of the first. Fix a non-empty, proper subset $S$ of $[m]$. For each $j \in S$, we set $S_{j} = S$, and for each $j \in [m] \setminus S$, we set $S_{j} = [m] \setminus S$. Thus P1 and Stopper are allowed to move only along edges $(u,v)$ such that \emph{either} $\sigma(u)$ and $\sigma(v)$ both belong to $S$, \emph{or} they both belong to $[m] \setminus S$, whereas P2 and Escaper are allowed to move along edges $(u,v)$ such that \emph{either} $\sigma(u) \in S$ and $\sigma(v) \in [m] \setminus S$, \emph{or} $\sigma(u) \in [m] \setminus S$ and $\sigma(v) \in S$.

\subsection{Notation}\label{subsec:notation}
Given a rooted tree $T$, we let $V(T)$ indicate the vertex set of $T$ and $\phi$ its root. Given $v \in V(T)$, we denote by $T(v)$ the subtree of $T$ comprising $v$ and all its descendants. For $n \in \mathbb{N}$, we let $[n]$ denote the set $\{1, \ldots, n\}$. 

Given tuples $(x_{1}, \ldots, x_{n})$ and $(y_{1}, \ldots, y_{n})$ in $\mathbb{R}^{n}$, for some $n \in \mathbb{N}$, we write $(x_{1}, \ldots, x_{n}) \preceq (y_{1}, \ldots, y_{n})$ if $x_{i} \leqslant y_{i}$ for each $i \in [n]$. Given a function $f: [0,1]^{n} \rightarrow [0,1]^{n}$, we denote by $\FP(f) = \left\{\mathbf{x} \in [0,1]^{n}: f(\mathbf{x}) = \mathbf{x}\right\}$ the set of all fixed points of $f$ in $[0,1]^{n}$. We define $\min \FP(f)$, if it exists, to be the unique $(x_{1}, \ldots, x_{n}) \in \FP(f)$ such that $(x_{1}, \ldots, x_{n}) \preceq (y_{1}, \ldots, y_{n})$ for all $(y_{1}, \ldots, y_{n}) \in \FP(f)$. Likewise, we define $\max \FP(f)$, if it exists, to be the unique $(x_{1}, \ldots, x_{n}) \in \FP(f)$ such that $(y_{1}, \ldots, y_{n}) \preceq (x_{1}, \ldots, x_{n})$ for all $(y_{1}, \ldots, y_{n}) \in \FP(f)$. We also denote by $f^{(n)}$ the $n$-fold composition of $f$ with itself. Given any real-valued $f$ defined on some domain $E$ of $\mathbb{R}^{n}$, we denote by $\partial_{i} f(x_{1}, \ldots, x_{n})$ the partial derivative $\frac{\partial}{\partial x_{i}} f(x_{1}, \ldots, x_{n})$ for each $i \in [n]$.

Given a non-empty, proper subset $S$ of $[n]$ and tuples $\mathbf{x}_{S} = (x_{i}: i \in S) \in \mathbb{R}^{|S|}$ and $\mathbf{y}_{[n] \setminus S}= (y_{i}: i \in [n] \setminus S) \in \mathbb{R}^{n-|S|}$, we let $\left(\mathbf{x}_{S} \vee \mathbf{y}_{[n] \setminus S}\right)$ denote their concatenation, i.e.\ the tuple $(z_{1}, \ldots, z_{n})$ where $z_{i} = x_{i}$ for each $i \in S$ and $z_{j} = y_{j}$ for each $j \in [n] \setminus S$. We denote by $\mathbf{1}_{S}$ the $|S|$-tuple in which each coordinate equals $1$, and by $\mathbf{0}_{S}$ the $|S|$-tuple in which each coordinate equals $0$.

For each $j \in [m]$, we let $G_{j}$ denote the pgf of $\chi_{j}$ (see Definition~\ref{defn:mult_GW}), i.e.\
\begin{equation}\label{eq:pgf}
G_{j}(x_{1}, \ldots, x_{m}) = \sum_{n_{1}, \ldots, n_{m} \in \mathbb{N}_{0}} \prod_{i=1}^{m} x_{i}^{n_{i}} \chi_{j}(n_{1}, \ldots, n_{m}), \text{ for all } (x_{1}, \ldots, x_{m}) \in [0,1]^{m}.
\end{equation}
Given any subset $S$ of $[m]$ and any $\mathbf{n}_{S} = \left(n_{k}: k \in S\right) \in \mathbb{N}_{0}^{|S|}$, we define the probability distribution
\begin{equation}\label{eq:dist_partial}
\chi_{j,S}\left(\mathbf{n}_{S}\right) = \Prob\left[X_{v,k} = n_{k} \text{ for all } k \in S\big|\sigma(v) = j\right], 
\end{equation}
where recall from Definition~\ref{defn:mult_GW} that $X_{v,k}$ indicates the number of children of $v$ that are of colour $k$ for any $k \in [m]$. We denote by $G_{j, S}$ the pgf corresponding to $\chi_{j,S}$. 

\subsection{Main results and organization of the paper}\label{subsec:main_results} 
The first of our main results expresses the probabilities of win / loss for each player, in each of the three games, as fixed points of suitable multivariable functions. These functions turn out to be compositions of translates of the pgfs $G_{j,S_{j}}$ and $G_{j,[m] \setminus S_{j}}$ for $j \in [m]$, where recall the subsets $S_{j}$ from Definition~\ref{defn:games}.

Before we state the theorem, we introduce the definitions of certain subsets of vertices of the tree $\mathcal{T} = \mathcal{T}_{[m] , \mathbf{p}, \bm{\chi}}$ described in Definition~\ref{defn:mult_GW}. For ease of exposition, we let `N', `M' and `E' indicate the normal, mis\`{e}re and escape games respectively, whereas `W', `L' and `D' denote, respectively, the outcomes of win, loss and draw for any particular player. For each $i = 1, 2$ and $j \in [m]$,
\begin{enumerate}
\item let $\NW_{i,j}$ comprise all $v \in V(\mathcal{T})$ with $\sigma(v) = j$, such that if $v$ is the initial vertex and P$i$ plays the first round of the normal game, then P$i$ wins;
\item let $\NL_{i,j}$ comprise all $v \in V(\mathcal{T})$ with $\sigma(v) = j$, such that if $v$ is the initial vertex and P$i$ plays the first round of the normal game, then P$i$ loses;
\item let $\ND_{i,j}$ comprise all $v \in V(\mathcal{T})$ with $\sigma(v) = j$, such that if $v$ is the initial vertex and P$i$ plays the first round, then the normal game results in a draw. 
\end{enumerate}
We analogously define $\MW_{i,j}$, $\ML_{i,j}$ and $\MD_{i,j}$ for the mis\`{e}re game. For the escape game, we let
\begin{enumerate}
\item $\ESW_{j}$ comprise all $v$ with $\sigma(v) = j$, such that if $v$ is the initial vertex and Stopper plays the first round, she wins;
\item $\EEL_{j}$ comprise all $v$ with $\sigma(v) = j$, such that if $v$ is the initial vertex and Escaper plays the first round, she loses.
\end{enumerate}

We define $\nw_{i,j}$ to be the probability that the root $\phi$ of $\mathcal{T}$ belongs to $\NW_{i,j}$. Analogously, we define $\nl_{i,j}$, $\nd_{i,j}$, $\mw_{i,j}$, $\ml_{i,j}$, $\md_{i,j}$, $\esw_{j}$ and $\eel_{j}$. We let $\bnw_{i} = (\nw_{i,j}: j \in [m])$, and likewise, $\bnl_{i}$, $\bmw_{i}$, $\bml_{i}$, $\besw$ and $\beel$. We also let $\alpha_{j}$ (respectively $\beta_{j}$) denote the probability that $\phi$ has no child of colour $k$ for any $k \in S_{j}$ (respectively any $k \in [m] \setminus S_{j}$), conditioned on $\sigma(\phi) = j$. 
\begin{theorem}\label{thm:main_1}
In case of the normal game, we have
\begin{equation}
\bnw_{1} = \min \FP(F_{N}) \quad \text{and} \quad \bnl_{1} = \mathbf{1}_{[m]} - \max \FP(F_{N}), 
\end{equation}
\begin{equation}
\bnw_{2} = \min \FP(\F_{N}) \quad \text{and} \quad \bnl_{2} = \mathbf{1}_{[m]} - \max \FP(\F_{N}),
\end{equation}
where the functions $F_{N}(x_{1}, \ldots, x_{m}) = \left(F_{N,j}(x_{1}, \ldots, x_{m}): j \in [m]\right)$ and $\F_{N}(x_{1}, \ldots, x_{m}) = \left(\F_{N,j}(x_{1}, \ldots, x_{m}): j \in [m]\right)$, mapping from $[0,1]^{m}$ to itself, are given by
\begin{multline}\label{F_{N,j}_script_F_{N,j}_defns}
F_{N,j}(x_{1}, \ldots, x_{m}) = 1 - G_{j, S_{j}}\left(1 - G_{k, [m] \setminus S_{k}}\left(x_{\ell}: \ell \in [m] \setminus S_{k}\right): k \in S_{j}\right) \\ \text{ and } \F_{N,j}(x_{1}, \ldots, x_{m}) = 1 - G_{j, [m] \setminus S_{j}}\left(1 - G_{k,S_{k}}\left(x_{\ell}: \ell \in S_{k}\right): k \in [m] \setminus S_{j}\right).
\end{multline}

In case of the mis\`{e}re game, we have 
\begin{equation}
\bmw_{1} = \min \FP(F_{M}) \quad \text{and} \quad \bml_{1} = \mathbf{1}_{[m]} - \max \FP(F_{M}),
\end{equation}
\begin{equation}
\bmw_{2} = \min \FP(\F_{M}) \quad \text{and} \quad \bml_{2} = \mathbf{1}_{[m]} - \max \FP(\F_{M}),
\end{equation}
where the functions $F_{M}(x_{1}, \ldots, x_{m}) = \left(F_{M,j}(x_{1}, \ldots, x_{m}): j \in [m]\right)$ and $\F_{M}(x_{1}, \ldots, x_{m}) = \left(\F_{M,j}(x_{1}, \ldots, x_{m}): j \in [m]\right)$, mapping from $[0,1]^{m}$ to itself, are given by
\begin{multline}\label{F_{M,j}_script_F_{M,j}_defns}
F_{M,j}(x_{1}, \ldots, x_{m}) = \alpha_{j} + 1 - G_{j,S_{j}}\left(1 - G_{k,[m] \setminus S_{k}}\left(x_{\ell}: \ell \in [m] \setminus S_{k}\right) + \beta_{k}: k \in S_{j}\right) \text{ and }\\ \F_{M,j} = \beta_{j} + 1 - G_{j,[m] \setminus S_{j}}\left(1 - G_{k,S_{k}}\left(x_{\ell}: \ell \in S_{k}\right) + \alpha_{k}: k \in [m] \setminus S_{j}\right).
\end{multline}

In case of the escape game, we have
\begin{equation}\label{escape_min_fixed_point}
\besw = \min \FP(F_{E}) \quad \text{and} \quad \beel = \mathbf{1}_{[m]} - \max \FP(\F_{E}),
\end{equation}
where the functions $F_{E}(x_{1}, \ldots, x_{m}) = \left(F_{E,j}(x_{1}, \ldots, x_{m}): j \in [m]\right)$ and $\F_{E}(x_{1}, \ldots, x_{m}) = \left(\F_{E,j}(x_{1}, \ldots, x_{m}): j \in [m]\right)$, mapping from $[0,1]^{m}$ to itself, are given by
\begin{multline}\label{F_{E,j}_script_F_{E,j}_defns}
F_{E,j}(x_{1}, \ldots, x_{m}) = \alpha_{j} + 1 - G_{j,S_{j}}\left(1 - G_{k, [m] \setminus S_{k}}\left(x_{\ell}: \ell \in [m] \setminus S_{k}\right): k \in S_{j} \right) \text{ and} \\ \F_{E,j}(x_{1}, \ldots, x_{m}) = 1 - G_{j, [m] \setminus S_{j}}\left(\alpha_{k} + 1 - G_{k,S_{k}}\left(x_{\ell}: \ell \in S_{k}\right): k \in [m] \setminus S_{j}\right).
\end{multline}
\end{theorem}

Our next couple of results pertain to two very fascinating examples: the \emph{bi-type binary GW tree} and the \emph{bi-type Poisson GW tree}. In each of these set-ups, the vertices get assigned one of the colours blue (indicated by $b$) and red (indicated by $r$). We set $S_{b} = \{b\}$ and $S_{r} = \{r\}$, whereby P1 and Stopper are allowed to move only along monochromatic directed edges, and P2 and Escaper only along non-monochromatic directed edges. 

The offspring distributions in case of the bi-type binary tree are described as follows:
\begin{enumerate}
\item given $\sigma(v) = b$, $v$ has no child with probability $p_{0}$, two blue children with probability $p_{\blbl}$, two red children with probability $p_{\rr}$, and one red and one blue child with probability $p_{\br}$,
\item given $\sigma(v) = r$, $v$ has no child with probability $q_{0}$, two blue children with probability $q_{\blbl}$, two red children with probability $q_{\rr}$, one red and one blue child with probability $q_{\br}$.
\end{enumerate}
In the bi-type Poisson tree, each vertex $v$, irrespective of its colour, gives birth to $\poi(\lambda)$ offspring in total, and
\begin{enumerate}
\item conditioned on $\sigma(v) = b$, a child of $v$ is assigned, independent of all else, colour $b$ with probability $p_{b}$, and colour $r$ with probability $p_{r} = 1-p_{b}$;
\item conditioned on $\sigma(v) = r$, a child of $v$ is assigned, independent of all else, colour $b$ with probability $q_{b}$, and colour $r$ with probability $q_{r} = 1-q_{b}$.
\end{enumerate}

\begin{theorem}\label{thm:main_example_1}
When the normal game is played on the bi-type binary GW tree, for each $i = 1, 2$, the draw probabilities $\nd_{i,b}$ and $\nd_{i,r}$ both equal $1$ if $p_{\br} = q_{\br} = 1$. Otherwise, $\nd_{i,b} = \nd_{i,r} = 0$. Analogous conclusions hold for the mis\`{e}re game. In case of the escape game, $\esl_{b} = \esl_{r} = 1$ when $p_{\br} = q_{\br} = 1$, and $\esl_{b} = \esl_{r} = 0$ otherwise.
\end{theorem} 

\begin{theorem}\label{thm:main_example_2}
When $p_{b}$ and $q_{r}$ are constants in $(0,1)$, the draw probabilities $\nd_{i,b}$ and $\nd_{i,r}$, for $i = 1, 2$, of the normal game on the bi-type Poisson GW tree can be brought arbitrarily close to $1$ by making $\lambda$ sufficiently large. Analogous conclusions hold for $\md_{i,b}$ and $\md_{i,r}$ in case of the mis\`{e}re game and for $\esl_{b}$ and $\esl_{r}$ in case of the escape game. If 
\begin{equation}\label{poisson_second_cond}
p_{b} p_{r} q_{b} q_{r} \leqslant \lambda^{-4} \exp\left\{\lambda p_{r} e^{-\lambda q_{r}} + \lambda p_{b} \exp\left\{-\lambda p_{r} \exp\left\{-\lambda q_{r} e^{-\lambda q_{b}}\right\}\right\} + \lambda q_{r} e^{-\lambda q_{b}}\right\},
\end{equation}
then the above probabilities are all $0$. If
\begin{equation}\label{normal_poisson_cond}
\lambda q_{r} e^{-\lambda q_{b}} \geqslant 1 \text{ and } \lambda p_{b} e^{-\lambda p_{r} e^{-\lambda q_{r} e^{-\lambda q_{b}}}} \geqslant 1, 
\end{equation}
then $\nd_{1,b} = 0$. Analogous results hold for $\nd_{1,r}$, $\nd_{2,b}$ and $\nd_{2,r}$. If
\begin{equation}\label{misere_poisson_cond}
\lambda q_{r} e^{-\lambda q_{b}} \geqslant 1 \text{ and } \lambda p_{b} e^{-\lambda p_{r} \left(1 - e^{-\lambda q_{r}}\right)} \geqslant 1,
\end{equation}
then $\md_{1,b} = 0$. Analogous results hold for $\md_{1,r}$, $\md_{2,b}$ and $\md_{2,r}$. If 
\begin{equation}\label{escape_poisson_cond}
\lambda q_{r} e^{-\lambda q_{b}} \geqslant 1 \text{ and } \lambda p_{b} e^{-\lambda p_{r} \left(e^{-\lambda q_{r} e^{-\lambda q_{b}}} - e^{-\lambda q_{r}}\right)} \geqslant 1,
\end{equation}
then $\esl_{b} = 0$. An analogous result holds for $\esl_{r}$. 
\end{theorem}

The next main result compares the probabilities of outcomes of different games with each other. 
\begin{theorem}\label{thm:main_2}
The following are true:
\begin{enumerate}
\item \label{main_2_part_1} For each $j \in [m]$, we have $\nw_{1,j} \leqslant \esw_{j}$, $\nl_{2,j} \leqslant \eel_{j}$, $\mw_{1,j} \leqslant \esw_{j}$ and $\ml_{2,j} \leqslant \eel_{j}$.

\item \label{main_2_part_2} When $G_{j,S_{j}}$ and $G_{j,[m] \setminus S_{j}}$ are differentiable for any $j \in [m]$, we have 
\begin{multline}\label{main_2_part_2_eq_1}
\esw_{j} \geqslant \alpha_{j} + \min\left\{\eel_{t}: t \in S_{j}\right\}\left(1 - \alpha_{j}\right) + \sum_{k \in S_{j}} \left(\eel_{k} - \min\left\{\eel_{t}: t \in S_{j}\right\}\right)\\ \partial_{k} G_{j,S_{j}}\left(1 - \eel_{t}: t \in S_{j}\right) \geqslant \min\left\{\eel_{t}: t \in S_{j}\right\};
\end{multline}
\begin{multline}
\mw_{1,j} \geqslant \alpha_{j} + \min\left\{\ml_{2,t}: t \in S_{j}\right\}\left(1 - \alpha_{j}\right) + \sum_{k \in S_{j}} \left(\ml_{2,k} - \min\left\{\ml_{2,t}: t \in S_{j}\right\}\right)\\ \partial_{k} G_{j,S_{j}}\left(1 - \ml_{2,t}: t \in S_{j}\right) \geqslant \min\left\{\ml_{2,t}: t \in S_{j}\right\}; 
\end{multline} 
\begin{multline}
\mw_{2,j} \geqslant \beta_{j} + \min\left\{\ml_{1,t}: t \in [m] \setminus S_{j}\right\}\left(1 - \beta_{j}\right) + \sum_{k \in [m] \setminus S_{j}} \left(\ml_{1,k} - \min\left\{\ml_{1,t}: t \in [m] \setminus S_{j}\right\}\right)\\ \partial_{k} G_{j,[m] \setminus S_{j}}\left(1 - \ml_{1,t}: t \in [m] \setminus S_{j}\right) \geqslant \min\left\{\ml_{1,t}: t \in [m] \setminus S_{j}\right\}. 
\end{multline} 

\item \label{main_2_part_3} Let $G_{j,S_{j}}$ be convex and continuously differentiable, $G_{j,[m] \setminus S_{j}}$ be continuous. If 
\begin{equation}\label{main_2_part_3_cond_1}
\alpha_{j} \leqslant \sum_{i \in S_{j}} \beta_{i} \Prob\left[X_{v,i} = 1, X_{v,k} = 0 \text{ for all } k \in S_{j} \setminus \{i\}\big|\sigma(v) = j\right], 
\end{equation}
for each $j \in [m]$, then $\bnl_{1} \preceq \bml_{1}$. If 
\begin{equation}\label{main_2_part_3_cond_2}
\alpha_{j} \geqslant \sum_{i \in S_{j}} \beta_{i} \E\left[X_{v,i}\big|\sigma(v) = j\right],
\end{equation}
then $\bnw_{1} \preceq \bmw_{1}$. Analogous inequalities hold for $\bnl_{2}$, $\bml_{2}$, $\bnw_{2}$ and $\bmw_{2}$.
\end{enumerate}
\end{theorem}
\begin{remark}\label{rem_intuitive}
Note that $\nl_{1,j}^{(2)} = \nl_{1,j}^{(1)} = \alpha_{j}$, since the only way that a normal game, with P1 playing the first round, ends in less than $2$ rounds \emph{and} in P1's loss, is if she fails to make the very first move, i.e.\ the root $\phi$, with $\sigma(\phi) = j$, has no child of colour $k$ for any $k \in S_{j}$. On the other hand, $\ml_{1,j}^{(2)} \geqslant \sum_{i \in S_{j}} \beta_{i} \Prob\left[X_{v,i} = 1, X_{v,k} = 0 \text{ for all } k \in S_{j} \setminus \{i\}\big|\sigma(v) = j\right]$, since if $\phi$ has one child $v$ of colour $k$ for $k \in S_{j}$, no child of colour $k'$ for any $k' \in S_{j} \setminus \{k\}$, and $v$ has no child of colour $\ell$ for any $\ell \in [m] \setminus S_{k}$, then P1 is forced to move the token to $v$ in the first round, and P2 wins immediately. Thus, \eqref{main_2_part_3_cond_1} guarantees $\bnl_{1}^{(2)} \preceq \bml_{1}^{(2)}$, which in turn, surprisingly, guarantees $\bnl_{1} \preceq \bml_{1}$. 

Likewise, we note that $\mw_{1,j}^{(2)} = \alpha_{j}$, since the root $\phi$ having no child of colour $k$ for any $k \in S_{j}$ and P1 playing the first round of the mis\`{e}re game implies that she fails to move in the first round and thus wins. On the other hand, if $\phi$ has at least one child $u$ with $\sigma(u) = k \in S_{j}$ and $u$ has no child of colour $\ell$ for any $\ell \in [m] \setminus S_{k}$, then in the normal game, P1 moves the token to such a $u$, and P2 is unable to move in the second round, thus yielding
\begin{align}
\nw_{1,j}^{(2)} &= \sum_{\mathbf{n}_{S_{j}} = \left(n_{k}: k \in S_{j}\right) \in \mathbb{N}_{0}^{|S_{j}|} \setminus \left\{\mathbf{0}_{S_{j}}\right\}} \left\{1 - \prod_{i \in S_{j}}\left(1 - \beta_{i}\right)^{n_{i}}\right\} \chi_{j}\left(\mathbf{n}_{S_{j}}\right),\nonumber
\end{align}
which is bounded above by $\sum_{i \in S_{j}} \beta_{i} \E\left[X_{v,i}\big|\sigma(v) = j\right]$, since $\left\{1 - \prod_{i \in S_{j}}\left(1 - \beta_{i}\right)^{n_{i}}\right\} \leqslant \sum_{i \in S_{j}} \beta_{i} n_{i}$. Thus, \eqref{main_2_part_3_cond_2} guarantees $\bnw_{1}^{(2)} \preceq \bmw_{1}^{(2)}$, which in turn ensures that $\bnw_{1} \preceq \bmw_{1}$. 
\end{remark}

Our next main result pertains to the behaviour of the probabilities of the game outcomes as functions of the law of $\mathcal{T}$. Given $\mathcal{T}_{\bm{\chi}} = \mathcal{T}_{[m], \mathbf{p}, \bm{\chi}}$, as defined in Definition~\ref{defn:mult_GW}), for a \emph{fixed} $m$, we denote its law by $\mathcal{L}_{\bm{\chi}}$. We define the distance between two such laws $\mathcal{L}_{\bm{\chi}}$ and $\mathcal{L}_{\bm{\eta}}$ as
\begin{equation}\label{metric_d_{0}}
d_{0}\left(\mathcal{L}_{\bm{\chi}}, \mathcal{L}_{\bm{\eta}}\right) = \max\left\{||\chi_{j} - \eta_{j}||_{\tv}: j \in [m]\right\}, 
\end{equation}
where $\tv$ indicates the total variation distance between two probability measures. 

We introduce a slight change of notation in Theorem~\ref{thm:main_3}, to emphasize the dependence of our objects of interest on the law of the GW process under consideration. For instance, on $T_{\bm{\chi}}$, we replace $\nw_{i,j}$ by $\nw_{i,j,\bm{\chi}}$, $G_{j,S}$ by $G_{j,S,\bm{\chi}}$, $F_{N,j}$ by $F_{N,j,\bm{\chi}}$, $\alpha_{j}$ and $\beta_{j}$ by $\alpha_{j,\bm{\chi}}$ and $\beta_{j,\bm{\chi}}$ respectively etc. Let $\Prob_{\mathcal{L}_{\bm{\chi}}}$ be the probability measure induced by $\mathcal{L}_{\bm{\chi}}$ and $\E_{\mathcal{L}_{\bm{\chi}}}$ expectation under $\Prob_{\mathcal{L}_{\bm{\chi}}}$. 

\begin{theorem}\label{thm:main_3}
Keeping $m$ fixed, define the following subsets of laws $\mathcal{L}_{\bm{\chi}}$:
\begin{equation}
\mathcal{D}_{1} = \left\{\mathcal{L}_{\bm{\chi}}: \alpha_{j,\bm{\chi}} > 0 \text{ for each } j \in [m]\right\}, \quad \mathcal{D}_{2} = \left\{\mathcal{L}_{\bm{\chi}}: \beta_{j,\bm{\chi}} > 0 \text{ for each } j \in [m]\right\};\nonumber
\end{equation}
\begin{multline}
\mathcal{D}_{3} = \left\{\mathcal{L}_{\bm{\chi}}: \mathcal{E}_{j,S_{j},\bm{\chi}} = \E_{\mathcal{L}_{\bm{\chi}}}\left[\sum_{k \in S_{j}} X_{v,k}\big|\sigma(v) = j\right] < \infty \text{ for each } j \in [m]\right\}, \\ \mathcal{D}_{4} = \left\{\mathcal{L}_{\bm{\chi}}: \mathcal{E}_{j,[m] \setminus S_{j},\bm{\chi}} = \E_{\mathcal{L}_{\bm{\chi}}}\left[\sum_{k \in [m] \setminus S_{j}} X_{v,k}\big|\sigma(v) = j\right] < \infty \text{ for each }j \in [m]\right\},\nonumber
\end{multline}
where recall from Definition~\ref{defn:mult_GW} that $X_{v,k}$ is the number of children coloured $k$ of a vertex $v$, for $k \in [m]$. Finally, let
\begin{multline}
\mathcal{C}_{1} = \left\{\mathcal{L}_{\bm{\chi}}: G_{j,S_{j},\bm{\chi}}\left(\beta_{k,\bm{\chi}}: k \in S_{j}\right) > \alpha_{j,\bm{\chi}} \text{ for each } j \in [m]\right\} \text{ and } \\ \mathcal{C}_{2} = \left\{\mathcal{L}_{\bm{\chi}}: G_{j,[m] \setminus S_{j},\bm{\chi}}\left(\alpha_{k,\bm{\chi}}: k\in [m] \setminus S_{j}\right) > \beta_{j,\bm{\chi}} \text{ for each } j \in [m]\right\}.\nonumber
\end{multline}
Then, the following are true, for each $i = 1, 2$ and $j \in [m]$:
\begin{enumerate}
\item \label{main_3_part_1} If $\mathcal{L}_{\bm{\chi}} \in \left(\mathcal{D}_{1} \cup \mathcal{D}_{4}\right) \cap \left(\mathcal{D}_{2} \cup \mathcal{D}_{3}\right)$, then $\nw_{i,j,\bm{\chi}}$ and $\nl_{i,j,\bm{\chi}}$ are lower semicontinuous functions of $\mathcal{L}_{\bm{\chi}}$ with respect to the metric $d_{0}$. If, moreover, $\nd_{i,j,\bm{\chi}} = 0$, then $\nw_{i,j,\bm{\chi}}$ and $\nl_{i,j,\bm{\chi}}$ are continuous functions of $\mathcal{L}_{\bm{\chi}}$.
\item \label{main_3_part_2} If $\mathcal{L}_{\bm{\chi}} \in \left(\mathcal{C}_{1} \cup \mathcal{D}_{4}\right) \cap \left(\mathcal{C}_{2} \cup \mathcal{D}_{3}\right)$, then $\mw_{i,j,\bm{\chi}}$ and $\ml_{i,j,\bm{\chi}}$ are lower semicontinuous functions of $\mathcal{L}_{\bm{\chi}}$ with respect to the metric $d_{0}$. If, moreover, $\md_{i,j,\bm{\chi}} = 0$, then $\mw_{i,j,\bm{\chi}}$ and $\ml_{i,j,\bm{\chi}}$ are continuous functions of $\mathcal{L}_{\bm{\chi}}$.
\item \label{main_3_part_3} If $\mathcal{L}_{\bm{\chi}} \in \mathcal{D}_{3} \cap \mathcal{D}_{4}$, then $\esw_{j,\bm{\chi}}$ and $\eel_{j,\bm{\chi}}$ are lower semicontinuous functions of $\mathcal{L}_{\bm{\chi}}$ with respect to the metric $d_{0}$. 
\end{enumerate}
\end{theorem}

Our final main result provides a sufficient condition that guarantees a positive probability for Escaper winning the escape game.
\begin{theorem}\label{thm:main_4} 
For every $i \in [m]$ and $j \in S_{i}$, let us define the probability $\gamma_{i,j} = \Prob\left[X_{v,j} = 1, X_{v,k} = 0 \text{ for all } k \in S_{i} \setminus \{j\}\big|\sigma(v) = i\right]$, and for $i, j \in [m]$, let $m_{i,j} = \E[X_{v,j}|\sigma(v) = i]$. If there exists a function $f: [m] \rightarrow [m]$ such that $f(i) \in S_{i}$ for every $i \in [m]$, and the matrix $M'' = \left(\left(m''_{i,j}\right)\right)_{i,j \in [m]}$, with $m''_{i,j} = \sum_{k \in [m] \setminus S_{i}: f(k) = j} m_{i,k} \gamma_{k,j}$, has its largest eigenvalue strictly greater than $1$, then Escaper has a positive probability of winning if she plays the first round, i.e.\ $\eew_{j} > 0$ for each $j \in [m]$. Moreover, as long as $\alpha_{j} < 1$ for each $j \in [m]$, we have $\eew_{j} > 0$ for all $j \in [m]$ iff $\esl_{j} > 0$ for all $j \in [m]$.
\end{theorem}

We now describe the organization of this paper. \S\ref{sec:proof_theorem_1} is dedicated to the proof of Theorem~\ref{thm:main_1}, with \S\ref{subsec:normal_recursions}, \S\ref{subsec:misere_recursions} and \S\ref{subsec:escape_recursions} respectively addressing the normal, mis\`{e}re and escape games. The proofs of Theorems~\ref{thm:main_example_1} and \ref{thm:main_example_2} are contained, respectively, in \S\ref{subsec:main_example_1_proof} and \S\ref{subsec:main_example_2_proof} of \S\ref{sec:main_examples_proof}. The proofs of \ref{main_3_part_1}, \ref{main_3_part_2} and \ref{main_3_part_3} of Theorem~\ref{thm:main_2} are given in three separate subsections \S\ref{subsec:main_2_part_1}, \S\ref{subsec:main_2_part_2} and \S\ref{subsec:main_2_part_3} of \S\ref{sec:main_2_proof}. Theorems~\ref{thm:main_3} and \ref{thm:main_4} are respectively proved in \S\ref{sec:main_3_proof} and \S\ref{sec:main_4_proof}.

\section{Proof of Theorem~\ref{thm:main_1}}\label{sec:proof_theorem_1}

\subsection{The normal games}\label{subsec:normal_recursions}
For every $n \in \mathbb{N}$, every $i =1, 2$ and each $j \in [m]$, let
\begin{enumerate}
\item $\NW_{i,j}^{(n)} \subset \NW_{i,j}$ comprise vertices $v$ such that if $v$ is the initial vertex and P$i$ plays the first round, the game lasts for less than $n$ rounds;
\item $\NL_{i,j}^{(n)} \subset \NL_{i,j}$ comprise vertices $v$ such that if $v$ is the initial vertex and P$i$ plays the first round, the game lasts for less than $n$ rounds;
\item $\ND_{i,j}^{(n)}$ comprise all vertices $v$ of $\mathcal{T}$ with $\sigma(v) = j$, such that $v \notin \NW_{i,j}^{(n)} \cup \NL_{i,j}^{(n)}$. 
\end{enumerate}
In other words, if $v \in \ND_{i,j}^{(n)}$ is the initial vertex and P$i$ plays the first round, the outcome of the game cannot be decided in less than $n$ rounds. We set $\NW_{i,j}^{(0)} = \NL_{i,j}^{(0)} = \emptyset$. We  define $\nw_{i,j}^{(n)}$, $\nl_{i,j}^{(n)}$ and $\nd_{i,j}^{(n)}$ to be the probabilities that the root $\phi$ belongs to $\NW_{i,j}^{(n)}$, $\NL_{i,j}^{(n)}$ and $\ND_{i,j}^{(n)}$ respectively, conditioned on $\sigma(\phi) = j$.

The following compactness result establishes that if a player is destined to win the normal game, she is able to do so in a finite number of rounds:
\begin{lemma}\label{lem:normal_compactness}
For each $i = 1, 2$ and $j \in [m]$, the subsets $\widetilde{\NW}_{i,j} = \NW_{i,j} \setminus \left(\bigcup_{n=1}^{\infty} \NW_{i,j}^{(n)}\right)$ and $\widetilde{\NL}_{i,j} = \NL_{i,j} \setminus \left(\bigcup_{n=1}^{\infty} \NL_{i,j}^{(n)}\right)$ are empty.
\end{lemma}
\begin{proof}
We only present the proof for $\widetilde{\NW}_{1,j}$, since the arguments to prove the rest of the claim are very similar. We show that, since P1 cannot guarantee to win the game in a finite number of rounds, P2 can actually ensure that the game results in a draw. For the initial vertex $v_{1}$ to be in $\widetilde{\NW}_{1,j}$, the following must be true:
\begin{itemize}
\item $\sigma(v_{1}) = j$;
\item $v_{1}$ cannot have any child $u$ in $\NL_{2,k}^{(n)}$ for any $k \in S_{j}$ and $n \in \mathbb{N}$, because in that case, P1 wins in less than $n+1$ rounds;
\item $v_{1}$ must have at least one child $v_{2}$ in $\widetilde{\NL}_{2,k}$ for some $k \in S_{j}$, and it is to such a $v_{2}$ that P1 moves the token in the first round, under optimal play.
\end{itemize}
Next, for $v_{2}$ to be in $\widetilde{\NL}_{2,k}$, the following must be true:
\begin{itemize}
\item $\sigma(v_{2}) = k$;
\item for every $\ell \in [m] \setminus S_{k}$, every child $u$ of $v_{2}$ with $\sigma(u) = \ell$ must be in $\NW_{1,\ell}$, since otherwise, P2 would not lose the game; 
\item $v_{2}$ must have at least one child $v_{3}$ in $\widetilde{\NW}_{1,\ell}$ for some $\ell \in [m] \setminus S_{k}$. This is because, for every $\ell \in [m] \setminus S_{k}$ and every child $u$ of $v_{2}$ with $\sigma(u) = \ell$, if we have $u \in \NW_{1,\ell}^{(n)}$ for some $n \in \mathbb{N}$, then no matter where P2 moves the token in the second round, she would lose the game in a finite number of rounds. Moreover, assuming optimal play, P2 would move the token to some $v_{3}$ in $\widetilde{\NW}_{1,\ell}$, for some $\ell \in [m] \setminus S_{k}$, in the second round. 
\end{itemize}
These observations reveal a pattern: the token starts at $v_{1}$ in $\widetilde{\NW}_{1,j}$, and for each $n \in \mathbb{N}$,
\begin{itemize}
\item before the $2n$-th round, it is at a vertex $v_{2n} \in \widetilde{\NL}_{2,\ell}$ for some $\ell \in S_{\sigma(v_{2n-1})}$;
\item before the $(2n+1)$-st round, it is at a vertex $v_{2n+1} \in \widetilde{\NW}_{1,\ell'}$ for some $\ell' \in [m] \setminus S_{\sigma(v_{2n})}$.
\end{itemize}
It is clear that the game never comes to an end, thereby ensuring a draw for both players. 
\end{proof}

The next lemma states a crucially used consequence of Lemma~\ref{lem:normal_compactness}:
\begin{lemma}\label{lem:normal_compactness_consequence}
For each $i = 1, 2$ and $j \in [m]$, we have $\nw_{i,j}^{(n)} \uparrow \nw_{i,j}$ and $\nl_{i,j}^{(n)} \uparrow \nl_{i,j}$ as $n \rightarrow \infty$.
\end{lemma}
\begin{proof}
By definition, it is immediate that $\NW_{i,j}^{(n)} \subseteq \NW_{i,j}^{(n+1)}$ and $\NL_{i,j}^{(n)} \subseteq \NL_{i,j}^{(n+1)}$. Consequently, $\left\{\NW_{i,j}^{(n)}\right\}_{n}$ and $\left\{\NL_{i,j}^{(n)}\right\}_{n}$ are increasing sequences of subsets. This, combined with Lemma~\ref{lem:normal_compactness}, gives us the conclusion of Lemma~\ref{lem:normal_compactness_consequence}.
\end{proof}

The rest of \S\ref{subsec:normal_recursions} is dedicated to deriving recursive relations and using these to prove the first part of Theorem~\ref{thm:main_1}. For a vertex $v$ to be in $\NW_{1,j}^{(n+1)}$, for any $j \in [m]$ and $n \in \mathbb{N}$, we must have $\sigma(v) = j$, and $v$ must have at least one child $u \in \NL_{2,k}^{(n)}$ for some $k \in S_{j}$. Therefore, we have
\begin{align}\label{normal_recur_1}
\nw_{1,j}^{(n+1)} &= \sum_{n_{r} \in \mathbb{N}_{0}: r \in [m]} \left\{1 - \prod_{k \in S_{j}} \left(1 - \nl_{2,k}^{(n)}\right)^{n_{k}}\right\} \chi_{j}(n_{1}, \ldots, n_{m}) = 1 - G_{j, S_{j}}\left(1 - \nl_{2,k}^{(n)}: k \in S_{j}\right).
\end{align}
We now verify that this recursion is true when $n = 0$. Since $\nl_{2,k}^{(0)} = 0$ for all $k \in S_{j}$, hence $G_{j, S_{j}}\left(1 - \nl_{2,k}^{(0)}: k \in S_{j}\right) = 1$. The right side of \eqref{normal_recur_1} thus equals $0$. For the fate of the game to be decided in less than $1$ round, P1 ought not to be able to make her very first move, but this is not possible if P1 is destined to win the game. Consequently, $\nw_{1,j}^{(1)} = 0$, thus showing that \eqref{normal_recur_1} holds for $n = 0$ as well. Likewise, for all $n \in \mathbb{N}_{0}$ and $j \in [m]$, we have
\begin{equation}\label{normal_recur_2}
\nw_{2,j}^{(n+1)} = 1 - G_{j, [m] \setminus S_{j}}\left(1 - \nl_{1,k}^{(n)}: k \in [m] \setminus S_{j}\right).
\end{equation}
For $v$ to be in $\NL_{1,j}^{(n+1)}$, for any $n \in \mathbb{N}$, we must have $\sigma(v) = j$, and every child $u$ of $v$ with $\sigma(u) = k$ for any $k \in S_{j}$ must be in $\NW_{2,k}^{(n)}$. Therefore
\begin{align}\label{normal_recur_3}
\nl_{1,j}^{(n+1)} &= \sum_{n_{r} \in \mathbb{N}_{0}: r \in [m]} \prod_{k \in S_{j}} \left(\nw_{2,k}^{(n)}\right)^{n_{k}} \chi_{j}(n_{1}, \ldots, n_{m}) = G_{j,S_{j}}\left(\nw_{2,k}^{(n)}: k \in S_{j}\right). 
\end{align}
Once again, we verify this recursion for $n = 0$. For the fate of the game to be decided in less than $1$ round, P1 must be unable to make her very first move, which happens only if $v$ does not have any child $u$ with $\sigma(u) = k$ for any $k \in S_{j}$. The probability of this event is equal to $G_{j,S_{j}}\left(\mathbf{0}_{S_{j}}\right) = G_{j,S_{j}}\left(\nw_{2,k}^{(0)}: k \in S_{j}\right)$, thus showing that \eqref{normal_recur_3} holds for $n = 0$. Likewise, for each $j \in [m]$ and $n \in \mathbb{N}_{0}$, we have
\begin{equation}\label{normal_recur_4}
\nl_{2,j}^{(n+1)} = G_{j, [m] \setminus S_{j}}\left(\nw_{1,k}^{(n)}: k \in [m] \setminus S_{j}\right).
\end{equation}
Using the above recursions, defining $F_{N}$, $\F_{N}$, $F_{N,j}$ and $\F_{N,j}$ as in \eqref{F_{N,j}_script_F_{N,j}_defns}, and setting $\bnw_{i}^{(n)} = \left(\nw_{i,j}^{(n)}: j \in [m]\right)$ for $i = 1, 2$, we get 
$\bnw_{1}^{(n+2)} = F_{N}\left(\bnw_{1}^{(n)}\right)$ and $\bnw_{2}^{(n+2)} = \F_{N}\left(\bnw_{2}^{(n)}\right)$ for all $n \in \mathbb{N}_{0}$. 
From Lemma~\ref{lem:normal_compactness_consequence}, we get 
\begin{equation}
\bnw_{1} = \lim_{n \rightarrow \infty} \bnw_{1}^{(2n)} = \lim_{n \rightarrow \infty} F_{N}^{(n)}\left(\bnw_{1}^{(0)}\right) = \lim_{n \rightarrow \infty} F_{N}^{(n)}\left(\mathbf{0}_{[m]}\right).\nonumber
\end{equation}
Thus, $\bnw_{1}$ is a fixed point of $F_{N}$. Likewise, $\bnw_{2}$ is a fixed point of $\F_{N}$ with $\bnw_{2} = \lim_{n \rightarrow \infty} \F_{N}^{(n)}\left(\mathbf{0}_{[m]}\right)$.

It can be easily verified that $G_{j,S}(x_{i}: i \in S) \leqslant G_{j,S}(y_{i}: i \in S)$ for any $S \subset [m]$ and tuples $(x_{i}: i \in S)$ and $(y_{i}: i \in S)$ in $[0,1]^{|S|}$ with $(x_{i}: i \in S) \preceq (y_{i}: i \in S)$. Thus, whenever $(x_{1}, \ldots, x_{m}) \preceq (y_{1}, \ldots, y_{m})$, we have $F_{N}(x_{1}, \ldots, x_{m}) \preceq F_{N}(y_{1}, \ldots, y_{m})$ and $\F_{N}(x_{1}, \ldots, x_{m}) \preceq \F_{N}(y_{1}, \ldots, y_{m})$. Therefore 
\begin{equation}\label{normal_min_fixed_point}
\bnw_{1} = \min \FP(F_{N}) \quad \text{and} \quad \bnw_{2} = \min \FP(\F_{N}).
\end{equation}

The above recursions also yield $\bnl_{1}^{(n+2)} = \mathbf{1}_{[m]} - F_{N}\left(\mathbf{1}_{[m]} - \bnl_{1}^{(n)}\right)$ and $\bnl_{2}^{(n+2)} = \mathbf{1}_{[m]} - \F_{N}\left(\mathbf{1}_{[m]} - \bnl_{2}^{(n)}\right)$, where we set $\bnl_{i}^{(n)} = \left(\nl_{i,j}^{(n)}: j \in [m]\right)$ for $i = 1, 2$. Lemma~\ref{lem:normal_compactness_consequence} then yields
\begin{align}\label{normal_loss_fixed_point_1}
\bnl_{1} = \lim_{n \rightarrow \infty} \bnl_{1}^{(2n)} = \lim_{n \rightarrow \infty} \mathbf{1}_{[m]} - F_{N}^{(n)}\left(\mathbf{1}_{[m]} - \bnl_{1}^{(0)}\right) = \mathbf{1}_{[m]} - \lim_{n \rightarrow \infty} F_{N}^{(n)}\left(\mathbf{1}_{[m]}\right), 
\end{align}
and likewise, $\bnl_{2} = \mathbf{1}_{[m]} - \lim_{n \rightarrow \infty} \F_{N}^{(n)}\left(\mathbf{1}_{[m]}\right)$. The observations made above on the monotonically increasing nature of $F_{N}$ and $\F_{N}$ allow us to conclude that
\begin{equation}\label{normal_max_fixed_point}
\bnl_{1} = \mathbf{1}_{[m]} - \max \FP(F_{N}) \quad \text{and} \quad \bnl_{2} = \mathbf{1}_{[m]} - \max \FP(\F_{N}).
\end{equation}

\subsection{The mis\`{e}re games}\label{subsec:misere_recursions}
We define the subsets $\MW_{i,j}^{(n)}$, $\ML_{i,j}^{(n)}$ and $\MD_{i,j}^{(n)}$ analogous to $\NW_{i,j}^{(n)}$, $\NL_{i,j}^{(n)}$ and $\ND_{i,j}^{(n)}$ respectively, and the conditional probabilities $\mw_{i,j}^{(n)}$, $\ml_{i,j}^{(n)}$ and $\md_{i,j}^{(n)}$ analogous to $\nw_{i,j}^{(n)}$, $\nl_{i,j}^{(n)}$ and $\nd_{i,j}^{(n)}$ respectively, for $i = 1, 2$, $j \in [m]$ and $n \in \mathbb{N}_{0}$.

We outline the proof of Lemma~\ref{lem:misere_compactness} since it is very similar to that of Lemma~\ref{lem:normal_compactness}: 
\begin{lemma}\label{lem:misere_compactness}
For each $i = 1, 2$ and $j \in [m]$, the subsets $\widetilde{\MW}_{i,j} = \MW_{i,j} \setminus \left(\bigcup_{n=1}^{\infty} \MW_{i,j}^{(n)}\right)$ and $\widetilde{\ML}_{i,j} = \ML_{i,j} \setminus \left(\bigcup_{n=1}^{\infty} \ML_{i,j}^{(n)}\right)$ are empty.
\end{lemma}
\begin{proof}
We show that $\widetilde{\MW}_{1,j}$ is empty by proving that P2 is able to ensure a draw instead of her loss. For the initial vertex $v_{1}$ to be in $\widetilde{\MW}_{i,j}$, it must have at least one child $u$ in $\ML_{2,k}$ for some $k \in S_{j}$, and no child of $v_{1}$ should be in $\ML_{2,k}^{(n)}$ for any $n \in \mathbb{N}$ and any $k \in S_{j}$. Thus, $v_{1}$ must have at least one child $v_{2}$ in $\widetilde{\ML}_{2,k}$ for some $k \in S_{j}$, and under optimal play, P1 moves the token to such a $v_{2}$ in the first round.

Next, for $v_{2}$ to be in $\widetilde{\ML}_{2,k}$, every child $u$ of $v_{2}$ with $\sigma(u) = \ell$ for any $\ell \in [m] \setminus S_{k}$ has to be in $\MW_{1,\ell}$, and $v_{2}$ must have at least one child $v_{3}$ in $\widetilde{\MW}_{1,\ell}$ for some $\ell \in [m] \setminus S_{k}$. Again, it is to such a $v_{3}$ that P2 moves the token under optimal play. These observations reveal the pattern along which the game proceeds, showing that the game never comes to an end, thus resulting in a draw.
\end{proof}

Similar to Lemma~\ref{lem:normal_compactness_consequence}, we conclude from Lemma~\ref{lem:misere_compactness} that as $n \rightarrow \infty$,
\begin{equation}\label{eq:misere_compactness_consequence}
\mw_{i,j}^{(n)} \uparrow \mw_{i,j} \text{ and } \ml_{i,j}^{(n)} \uparrow \ml_{i,j} \text{ for each } i = 1, 2 \text{ and } j \in [m].
\end{equation}

The derivation of recursion relations is quite similar to those of \S\ref{subsec:normal_recursions}. For $j \in [m]$ and $n \in \mathbb{N}$, an initial vertex $v$ is in $\MW_{1,j}^{(n+1)}$ if either $v$ has no child $u$ with $\sigma(u) = k$ for any $k \in S_{j}$, or $v$ has at least one child $u$ in $\ML_{2,k}^{(n)}$ for some $k \in S_{j}$. Recalling $\alpha_{j} = G_{j,S_{j}}\left(\mathbf{0}_{S_{j}}\right)$ from \S\ref{subsec:main_results}, we get
\begin{multline}\label{misere_recur_1}
\mw_{1,j}^{(n+1)} = \alpha_{j} + \sum_{n_{r} \in \mathbb{N}_{0}: r \in [m]} \left\{1 - \prod_{k \in S_{j}}\left(1 - \ml_{2,k}^{(n)}\right)^{n_{k}}\right\} \chi_{j}(n_{1}, \ldots, n_{m}) \\= \alpha_{j} + 1 - G_{j,S_{j}}\left(1 - \ml_{2,k}^{(n)}: k \in S_{j}\right).
\end{multline}
We now verify this recursion for $n = 0$. For the outcome to be decided in less than $1$ round, P1 must be unable to make her very first move, thus winning the game immediately. This happens only if $v$ has no child $u$ with $\sigma(u) = k$ for any $k \in S_{j}$, and this event has probability $\alpha_{j}$. Thus $\mw_{1,j}^{(1)} = \alpha_{j}$. On the other hand, since $\ml_{2,k}^{(0)} = 0$ for each $k \in S_{j}$, we get $G_{j,S_{j}}\left(1 - \ml_{2,k}^{(0)}: k \in S_{j}\right) = 1$, making the right side of \eqref{misere_recur_1} equal $\alpha_{j}$ for $n = 0$. This concludes the verification. Likewise, for each $j \in [m]$ and $n \in \mathbb{N}_{0}$, we deduce that
\begin{equation}
\mw_{2,j}^{(n+1)} = \beta_{j} + 1 - G_{j,[m] \setminus S_{j}}\left(1 - \ml_{1,k}^{(n)}: k \in [m] \setminus S_{j}\right). \nonumber
\end{equation}
For an initial vertex $v$ to be in $\ML_{1,j}^{(n+1)}$ for $n \in \mathbb{N}$, it must have at least one child $u$ with $\sigma(u) = k$ for some $k \in S_{j}$, and every child $u$ with $\sigma(u) = k$ for any $k \in S_{j}$ must be in $\MW_{2,k}^{(n)}$. Thus,
\begin{align}\label{misere_recur_3}
\ml_{1,j}^{(n+1)} &= \sum_{\substack{n_{r} \in \mathbb{N}_{0}: r \in [m]\\(n_{k}: k \in S_{j}) \neq \mathbf{0}_{S_{j}}}} \prod_{k \in S_{j}}\left(\mw_{2,k}^{(n)}\right)^{n_{k}} \chi_{j}(n_{1}, \ldots, n_{m}) = G_{j,S_{j}}\left(\mw_{2,k}^{(n)}: k \in S_{j}\right) - \alpha_{j}.
\end{align}
Once again, we verify this recursion for $n = 0$. For the outcome to be decided in less than $1$ round, P1 must be unable to make her first move, allowing her to win. Since this is not an option, i.e.\ P1 must lose, hence $\ml_{1,j}^{(1)} = 0$. On the other hand, since $\mw_{2,k}^{(0)} = 0$ for each $k \in S_{j}$, hence $G_{j,S_{j}}\left(\mw_{2,k}^{(0)}: k \in S_{j}\right) = \alpha_{j}$, thus making the right side of \eqref{misere_recur_3} equal $0$ as well. Likewise, for each $j \in [m]$ and $n \in \mathbb{N}_{0}$, we have
\begin{equation}
\ml_{2,j}^{(n+1)} = G_{j,[m] \setminus S_{j}}\left(\mw_{1,k}^{(n)}: k \in [m] \setminus S_{j}\right) - \beta_{j}.\nonumber
\end{equation}

Using these recursions, defining $F_{M}$, $\F_{M}$ , $F_{M,j}$ and $\F_{M,j}$ as in \eqref{F_{M,j}_script_F_{M,j}_defns}, and setting $\bmw_{i}^{(n)} = \left(\mw_{i,j}^{(n)}: j \in [m]\right)$ for $i = 1, 2$, we get $\bmw_{1}^{(n+2)} = F_{M}\left(\bmw_{1}^{(n)}\right)$ and $\bmw_{2}^{(n+2)} = \F_{M}\left(\bmw_{2}^{(n)}\right)$. Using \eqref{eq:misere_compactness_consequence}, we get $\bmw_{1} = \lim_{n \rightarrow \infty} F_{M}^{(n)}\left(\mathbf{0}_{[m]}\right)$ and $\bmw_{2} = \lim_{n \rightarrow \infty} \F_{M}^{(n)}\left(\mathbf{0}_{[m]}\right)$, and utilizing the monotonically increasing nature of both $F_{M}$ and $\F_{M}$, we get
\begin{equation}\label{misere_min_fixed_point}
\bmw_{1} = \min \FP(F_{M}) \quad \text{and} \quad \bmw_{2} = \min \FP(\F_{M}).
\end{equation} 

The above recursions also yield $\bml_{1}^{(n+2)} = \mathbf{1}_{[m]} - F_{M}\left(\mathbf{1}_{[m]} - \bml_{1}^{(n)}\right)$ and $\bml_{2}^{(n+2)} = \mathbf{1}_{[m]} - \F_{M}\left(\mathbf{1}_{[m]} - \bml_{2}^{(n)}\right)$, where $\bml_{i}^{(n)} = \left(\ml_{i,j}^{(n)}: j \in [m]\right)$ for $i = 1, 2$, so that \eqref{eq:misere_compactness_consequence} leads to 
\begin{equation}\label{misere_max_fixed_point}
\bml_{1} = \mathbf{1}_{[m]} - \max \FP(F_{M}) \quad \text{and} \quad \bml_{2} = \mathbf{1}_{[m]} - \max \FP(\F_{M}).
\end{equation}

\subsection{The escape games}\label{subsec:escape_recursions}
For $j \in [m]$ and $n \in \mathbb{N}$, we let $\ESW_{j}^{(n)} \subset \ESW_{j}$ comprise vertices $v$ such that if $v$ is the initial vertex and Stopper plays the first round, she wins in less than $n$ rounds. We let $\EEL_{j}^{(n)} \subset \EEL_{j}$ comprise vertices $v$ such that if $v$ is the initial vertex and Escaper plays the first round, she loses in less than $n$ rounds. We set $\ESW_{j}^{(0)} = \EEL_{j}^{(0)} = \emptyset$. We let $\ESL_{j}^{(n)}$ and $\EEW_{j}^{(n)}$ be the complement subsets of $\ESW_{j}^{(n)}$ and $\EEL_{j}^{(n)}$ respectively.  

\begin{lemma}\label{lem:escape_compactness}
For each $j \in [m]$, the subsets $\widetilde{\ESW}_{j} = \ESW_{j} \setminus \left(\bigcup_{n=1}^{\infty}\ESW_{j}^{(n)}\right)$ and $\widetilde{\EEL}_{j} = \EEL_{j} \setminus \left(\bigcup_{n=1}^{\infty}\EEL_{j}^{(n)}\right)$ are empty.
\end{lemma}
\begin{proof}
We prove the claim for $\widetilde{\ESW}_{j}$. Given a realization $T$ of $\mathcal{T}$, we construct a corresponding rooted tree $T'$ so as to draw a parallel between the escape game played on $T$ and a suitable normal game on $T'$. Let $w$ be the initial vertex for the escape game on $T$. We denote the root of $T'$ by $w_{e}$. For every vertex $v$ of $T(w) \setminus \{w\}$, we create two copies of $v$ in $T'$ -- the odd copy $v_{o}$ and the even copy $v_{e}$, and we assign colours $\sigma'(v_{o}) = \sigma'(v_{e}) = \sigma(v)$, where $\sigma(v)$ is the colour of $v$ in $T$. If $u$ is the parent of $v$ in $T(w)$, then in $T'$, we include only the directed edges $(u_{o}, v_{1})$ and $(u_{1}, v_{o})$. If $v$ in $T(w)$ is at an even distance from $w$ and $v$ has no child of colour $k$ for any $k \in S_{\sigma(v)}$, we call $v$ \emph{special}. In this case, we choose some $k \in S_{j}$, add a vertex $\nu_{v}$ to $T'$ with $\sigma'(\nu_{v}) = k$, add the edge $\left(v_{e}, \nu_{v}\right)$, and keep $\nu_{v}$ childless. 

The normal game on $T'$ begins at $w_{e}$ and P1 plays the first round. The permissible edges for P1 are the same as those for Stopper, and the permissible edges for P2 are the same as those for Escaper. Stopper executes the following strategy. If on $T'$, P1 moves the token from $u_{e}$ to $v_{o}$, and $u$ is not a special vertex in $T$, then in the corresponding round of the escape game, Stopper moves the token from $u$ to $v$. If $u$ is a special vertex, then P1 is able to move the token from $u_{e}$ to $\nu_{u}$, but Stopper fails to make a move and the escape game comes to an end.

We argue that Stopper wins the escape game on $T$ iff P1 wins the normal game on $T'$. If Stopper wins because of being unable to make a move, then the escape game must have reached, at the end of an even round, a special vertex $u$ in $T(w)$. By our construction, P1 moves the token, in that round, from $u_{e}$ to $\nu_{u}$, but in the next round, P2 gets stuck since $\nu_{u}$ is childless. If Stopper wins because of Escaper being unable to make a move, then the escape game must have reached, at the end of an odd round, a vertex $u$ in $T(w)$ such that $u$ has no child $v$ with $\sigma(v) \in [m] \setminus S_{\sigma(u)}$. In this case, the normal game, at the end of the corresponding round, must have reached $u_{o}$, and there exists no child $v_{e}$ of $u_{o}$ with $\sigma'(v_{e}) \in [m] \setminus S_{\sigma'(u_{o})}$. Consequently, P2 fails to move in this round and loses. Thus $\widetilde{ESW}_{j}$ on $T$ corresponds to $\widetilde{\NW}_{1,j}$ on $T'$. By Lemma~\ref{lem:normal_compactness}, the conclusion follows.
\end{proof}

Letting $\esw_{j}^{(n)}$ and $\eel_{j}^{(n)}$ denote the probabilities, conditioned on $\sigma(\phi) = j$, that $\phi$ belongs to $\ESW_{j}^{(n)}$ and $\EEL_{j}^{(n)}$ respectively, for each $n \in \mathbb{N}_{0}$, we conclude, as in Lemma~\ref{lem:normal_compactness_consequence}, that as $n \rightarrow \infty$,
\begin{equation}\label{eq:escape_compactness_consequence}
\esw_{j}^{(n)} \uparrow \esw_{j} \text{ and } \eel_{j}^{(n)} \uparrow \eel_{j} \text{ for each } j \in [m].
\end{equation}

For an initial vertex $v$ to be in $\ESW_{j}^{(n+1)}$, for  $n \in \mathbb{N}$, either $v$ has no child of colour $k$ for any $k \in S_{j}$, or $v$ has at least one child $u$ in $\EEL_{k}^{(n)}$ for some $k \in S_{j}$. Thus, 
\begin{multline}\label{escape_recur_1}
\esw_{j}^{(n+1)} = \alpha_{j} + \sum_{n_{r} \in \mathbb{N}_{0}: r \in [m]} \left\{1 - \prod_{k \in S_{j}}\left(1 - \eel_{k}^{(n)}\right)^{n_{k}}\right\} \chi_{j}(n_{1}, \ldots, n_{m}) \\= \alpha_{j} + 1 - G_{j,S_{j}}\left(1 - \eel_{k}^{(n)}: k \in S_{j}\right).
\end{multline}
We now verify this recursion for $n = 0$. For Stopper to win the game in less than $1$ round, she must fail to make a move in the very first round, which happens only if $v$ has no child of colour $k$ for any $k \in S_{j}$. The probability of this event is $\alpha_{j}$. Hence, $\esw_{j}^{(1)} = \alpha_{j}$. On the other hand, $1 - G_{j,S_{j}}\left(1 - \eel_{k}^{(0)}: k \in S_{j}\right) = 0$ since $\eel_{k}^{(0)} = 0$ for each $k \in S_{j}$. This concludes the verification. For $v$ to be in $\EEL_{j}^{(n+1)}$ for some $n \in \mathbb{N}$, either $v$ has no child of colour $k$ for any $k \in [m] \setminus S_{j}$, or, for every $k \in [m] \setminus S_{j}$, every child $u$ of $v$ with $\sigma(u) = k$ is in $\ESW_{k}^{(n)}$. Thus 
\begin{multline}\label{escape_recur_2}
\eel_{j}^{(n+1)} = \beta_{j} + \sum_{\substack{n_{r} \in \mathbb{N}_{0}: r \in [m]\\\left(n_{k}: k \in [m] \setminus S_{j}\right) \neq \mathbf{0}_{[m] \setminus S_{j}}}} \prod_{k \in [m] \setminus S_{j}}\left(\esw_{k}^{(n)}\right)^{n_{k}} \chi_{j}(n_{1}, \ldots, n_{m})\\ = G_{j, [m] \setminus S_{j}}\left(\esw_{k}^{(n)}: k \in [m] \setminus S_{j}\right). 
\end{multline}
Once again, we verify this recursion for $n = 0$. For Escaper to lose the game in less than $1$ round, she must be unable to make her very first move, which means that $v$ has no child of colour $k$ for any $k \in [m] \setminus S_{j}$. This event happens with probability $\beta_{j}$. Thus, $\eel_{j}^{(1)} = \beta_{j}$. On the other hand, $G_{j, [m] \setminus S_{j}}\left(\esw_{k}^{(0)}: k \in [m] \setminus S_{j}\right) = G_{j, [m] \setminus S_{j}}\left(\mathbf{0}_{[m] \setminus S_{j}}\right) = \beta_{j}$.

Using these recursions, defining $F_{E}$, $\F_{E}$, $F_{E,j}$ and $\F_{E,j}$ as in \eqref{F_{E,j}_script_F_{E,j}_defns}, we get $\besw^{(n+2)} = F_{E}\left(\besw^{(n)}\right)$ and $\beel^{(n+2)} = \mathbf{1}_{[m]} - \F_{E}\left(\mathbf{1}_{[m]} - \beel^{(n)}\right)$, where $\besw^{(n)} = \left(\esw_{j}^{(n)}: j \in [m]\right)$ and $\beel^{(n)} = \left(\eel_{j}^{(n)}: j \in [m]\right)$. Using the monotonically increasing nature of $F_{E}$ and $\F_{E}$ and \eqref{eq:escape_compactness_consequence}, 
\begin{equation}\label{escape_min_fixed_point}
\besw = \lim_{n \rightarrow \infty} \besw^{(2n)} = \lim_{n \rightarrow \infty} F_{E}^{(n)}\left(\mathbf{0}_{[m]}\right) = \min \FP(F_{E}),
\end{equation}
\begin{equation}\label{escape_max_fixed_point}
\beel = \lim_{n \rightarrow \infty} \beel^{(2n)} = \mathbf{1}_{[m]} - \lim_{n \rightarrow \infty} \F_{E}^{(n)}\left(\mathbf{1}_{[m]} - \beel^{(0)}\right) = \mathbf{1}_{[m]} - \max \FP(\F_{E}).
\end{equation}

\section{Proofs of Theorems~\ref{thm:main_example_1} and \ref{thm:main_example_2}}\label{sec:main_examples_proof}
\subsection{Proof of Theorem~\ref{thm:main_example_1}}\label{subsec:main_example_1_proof} The generating functions involved in this example are
\begin{align}
& G_{b,S_{b}}(x_{k}: k \in S_{b}) = G_{b,\{b\}}(x_{b}) = (p_{0}+p_{\rr}) + p_{\br} x_{b} + p_{\blbl} x_{b}^{2};\nonumber\\
& G_{r,S_{r}}(x_{k}: k \in S_{r}) = G_{r,\{r\}}(x_{r}) = (q_{0}+q_{\blbl}) + q_{\br} x_{r} + q_{\rr} x_{r}^{2}; \nonumber\\
& G_{b, [m] \setminus S_{b}}(x_{k}: k \in [m] \setminus S_{b}) = G_{b,\{r\}}(x_{r}) = (p_{0} + p_{\blbl}) + p_{\br} x_{r} + p_{\rr} x_{r}^{2};\nonumber\\
& G_{r, [m] \setminus S_{r}}(x_{k}: k \in [m] \setminus S_{r}) = G_{r,\{b\}}(x_{b}) = (q_{0}+q_{\rr}) + q_{\br} x_{b} + q_{\blbl} x_{b}^{2}. \nonumber
\end{align}
From these and the recursions derived in \S\ref{subsec:normal_recursions}, we get
\begin{align}
\nw_{1,b} &= p_{\br} + p_{\blbl} - p_{\br}\left(p_{\br} + p_{\rr} - p_{\br} \nw_{1,r} - p_{\rr} \nw_{1,r}^{2}\right) - p_{\blbl} \left(p_{\br} + p_{\rr} - p_{\br} \nw_{1,r} - p_{\rr} \nw_{1,r}^{2}\right)^{2} \nonumber\\
&= (p_{0}+p_{\blbl})\left\{p_{\br} + 2p_{\blbl} - p_{\blbl}(p_{0}+p_{\blbl})\right\} + \left(p_{br}^{2} + 2p_{\br}^{2}p_{\blbl} + 2p_{\br}p_{\blbl}p_{\rr}\right)\nw_{1,r} \nonumber\\& + \left(p_{\br}p_{\rr} - p_{\blbl}p_{\br}^{2} + 2p_{\blbl}p_{\rr}p_{\br} + 2p_{\blbl}p_{\rr}^{2}\right)\nw_{1,r}^{2} - 2p_{\blbl}p_{\br}p_{\rr}\nw_{1,r}^{3} - p_{\blbl}p_{\rr}^{2}\nw_{1,r}^{4}.  \nonumber
\end{align}
On the other hand, we have
\begin{multline}
1 - \nl_{1,b} = (p_{0}+p_{\blbl})\left\{p_{\br} + 2p_{\blbl} - p_{\blbl}(p_{0}+p_{\blbl})\right\} + \left(p_{br}^{2} + 2p_{\br}^{2}p_{\blbl} + 2p_{\br}p_{\blbl}p_{\rr}\right)\left(1 - \nl_{1,r}\right) +\\ \left(p_{\br}p_{\rr} - p_{\blbl}p_{\br}^{2} + 2p_{\blbl}p_{\rr}p_{\br} + 2p_{\blbl}p_{\rr}^{2}\right)\left(1 - \nl_{1,r}\right)^{2} - 2p_{\blbl}p_{\br}p_{\rr}\left(1 - \nl_{1,r}\right)^{3} - p_{\blbl}p_{\rr}^{2}\left(1 - \nl_{1,r}\right)^{4}.\nonumber
\end{multline}
Combined together, these yield
\begin{multline}\label{nd_{1,b}_recursion_example_4}
\nd_{1,b} = \nd_{1,r}\Big[p_{br}^{2} + 2p_{\br}^{2}p_{\blbl} + 2p_{\br}p_{\blbl}p_{\rr} + \left(p_{\br}p_{\rr} - p_{\blbl}p_{\br}^{2} + 2p_{\blbl}p_{\rr}p_{\br} + 2p_{\blbl}p_{\rr}^{2}\right) \left(1 - \nl_{1,r} + \nw_{1,r}\right) -\\ 2p_{\blbl}p_{\br}p_{\rr}\left\{\left(1 - \nl_{1,r}\right)^{2} + \nw_{1,r}^{2} + \nw_{1,r}\left(1 - \nl_{1,r}\right)\right\} - p_{\blbl}p_{\rr}^{2}\left(1 - \nl_{1,r} + \nw_{1,r}\right)\left\{\left(1 - \nl_{1,r}\right)^{2} + \nw_{1,r}^{2}\right\}\Big].
\end{multline}
By symmetry, we have
\begin{multline}\label{nd_{1,r}_recursion_example_4}
\nd_{1,r} = \nd_{1,b}\Big[\left(q_{br}^{2} + 2q_{\br}^{2}q_{\rr} + 2q_{\br}q_{\rr}q_{\blbl}\right) + \left(q_{\br}q_{\blbl} - q_{\rr}q_{\br}^{2} + 2q_{\rr}q_{\blbl}q_{\br} + 2q_{\rr}q_{\blbl}^{2}\right) \left(1 - \nl_{1,b} + \nw_{1,b}\right) \\- 2q_{\rr}q_{\br}q_{\blbl}\left\{\left(1 - \nl_{1,b}\right)^{2} + \nw_{1,b}^{2} + \nw_{1,b}\left(1 - \nl_{1,b}\right)\right\} - q_{\rr}q_{\blbl}^{2}\left(1 - \nl_{1,b} + \nw_{1,b}\right)\left\{\left(1 - \nl_{1,b}\right)^{2} + \nw_{1,b}^{2}\right\}\Big].
\end{multline}
We now split our analysis into two parts: the first is where  $A = p_{\br}p_{\rr} - p_{\blbl}p_{\br}^{2} + 2p_{\blbl}p_{\rr}p_{\br} + 2p_{\blbl}p_{\rr}^{2}$ from \eqref{nd_{1,b}_recursion_example_4} is non-negative, and the second is where $A$ is negative. When $A \geqslant 0$, from \eqref{nd_{1,b}_recursion_example_4} and using $p_{\rr} \leqslant 1 - p_{\br} - p_{\blbl}$, we get
\begin{align}\label{A_nonnegative_upper_bound}
\nd_{1,b} &\leqslant \nd_{1,r}\left[\left(p_{br}^{2} + 2p_{\br}^{2}p_{\blbl} + 2p_{\br}p_{\blbl}p_{\rr}\right) + 2\left(p_{\br}p_{\rr} - p_{\blbl}p_{\br}^{2} + 2p_{\blbl}p_{\rr}p_{\br} + 2p_{\blbl}p_{\rr}^{2}\right)\right] \nonumber\\
&\leqslant \nd_{1,r}\left(p_{\br}^{2} + 6p_{\blbl}p_{\br}p_{\rr} + 2p_{\br}p_{\rr} + 4p_{\blbl}p_{\rr}^{2}\right) \nonumber\\
&\leqslant \nd_{1,r}\left(2p_{\br} + 4p_{\blbl} - p_{\br}^{2} - 8p_{\blbl}^{2} - 4p_{\br}p_{\blbl} - 2p_{\br}^{2}p_{\blbl} + 2p_{\br}p_{\blbl}^{2} + 4p_{\blbl}^{3}\right).
\end{align}
Define the domain $\mathcal{D} = \left\{(x,y): x \in [0,1], y \in [0,1], x+y \leqslant 1\right\}$ and the function $f: \mathcal{D} \rightarrow \mathbb{R}$ as
\begin{equation}
f(x,y) = 2x + 4y - x^{2} - 8y^{2} - 4xy - 2x^{2}y + 2xy^{2} + 4y^{3}.\nonumber
\end{equation}
The partial derivatives of this function are $\frac{\partial}{\partial x} f(x,y) = 2 - 2x - 4y - 4xy + 2y^{2}$ and $\frac{\partial}{\partial y}f(x,y) = 4 - 16y - 4x - 2x^{2} + 4xy + 12y^{2}$. From the equation $\frac{\partial}{\partial x}f(x,y) = 0$, we have
\begin{equation}\label{x_value_example_4_normal}
2x + 4xy = 2 - 4y + 2y^{2} \implies x = \frac{(1-y)^{2}}{1+2y}.
\end{equation}
Substituting this in the equation $\frac{\partial}{\partial y}f(x,y) = 0$, we get the quartic polynomial equation $54y^{4} - 28y^{3} - 36y^{2} + 12y = 2$, with real roots $y_{1} = 1$ and $y_{2} \approx -0.77985$, of which only $y_{1}$ lies in our domain of interest. The corresponding value of $x$, from \eqref{x_value_example_4_normal}, is $x_{1} = 0$, and the value of $f$ at this critical point is $f(0,1) = 0$. On the boundary of $\mathcal{D}$:
\begin{itemize}
\item When $x+y = 1$, we have $f(x,y) =  x^{2}$ strictly increasing on $[0,1]$ with maximum $f(1,0) = 1$. 
\item When $x = 0$, we have $f(0,y) = 4y - 8y^{2} + 4y^{3} = 4y(1-y)^{2}$, so that $f'(0,y) = 4 - 16y + 12y^{2} = 4(1-y)(1-3y)$ is strictly positive only if $y < \frac{1}{3}$. Thus $f(0,y)$ is strictly increasing on $\left[0, \frac{1}{3}\right)$ and strictly decreasing on $\left(\frac{1}{3}, 1\right]$, and the maximum is $f\left(0, \frac{1}{3}\right) = \frac{16}{27}$.
\item When $y = 0$, we have $f(x,0) = 2x - x^{2}$, so that $f'(x,0) = 2 - 2x = 2(1-x)$ is strictly positive for $x < 1$, and the maximum, attained at $x = 1$, is $f(1,0) = 1$.
\end{itemize}
These observations, together with \eqref{A_nonnegative_upper_bound}, imply that when $A \geqslant 0$, we have $p_{\br}^{2} + 6p_{\blbl}p_{\br}p_{\rr} + 2p_{\br}p_{\rr} + 4p_{\blbl}p_{\rr}^{2} < 1$ for all $p_{\br}$, $p_{\blbl}$ and $p_{\rr}$, except when $p_{\br} = 1$. When $A$ is non-positive, \eqref{nd_{1,b}_recursion_example_4} yields
\begin{multline}\label{A_negative_upper_bound}
\nd_{1,b} \leqslant \nd_{1,r}\left(p_{br}^{2} + 2p_{\br}^{2}p_{\blbl} + 2p_{\br}p_{\blbl}p_{\rr}\right) = \nd_{1,r}\left[p_{\br}^{2} + 2p_{\br} p_{\blbl}(p_{\br} + p_{\rr})\right] \\ \leqslant \nd_{1,r}\left[p_{\br}^{2} + 2p_{\br} p_{\blbl}\right] \leqslant \nd_{1,r}\left(p_{\br} + p_{\blbl}\right)^{2} \leqslant \nd_{1,r},
\end{multline}
and the only scenario in which this inequality is an equality is where $p_{\br} = 1$. 

We conclude that unless $p_{\br} = 1$, we have $\nd_{1,b} < \nd_{1,r}$ if we assume that $\nd_{1,r}$ is strictly positive. A similar analysis on \eqref{nd_{1,r}_recursion_example_4} shows that unless $q_{\br} = 1$, we have $\nd_{1,r} < \nd_{1,b}$ if we assume that $\nd_{1,b}$ is strictly positive. Clearly, these two inequalities cannot hold simultaneously. Therefore, we have:
\begin{itemize} 
\item $\nd_{1,b} = \nd_{1,r} = 1$ when $p_{\br} = q_{\br} = 1$,
\item and in all other cases, $\nd_{1,b} = \nd_{1,r} = 0$.
\end{itemize}

For the mis\`{e}re game, $\alpha_{b} = p_{0} + p_{\rr}$ and $\beta_{b} = p_{0} + p_{\blbl}$. From the generating functions above and the recursions in \S\ref{subsec:misere_recursions} we get
\begin{multline}
\mw_{1,b} = p_{0} + p_{\rr} + \left(p_{\br}^{2} + 2p_{\blbl}p_{\br}\right) \mw_{1,r} + \left(p_{\br}p_{\rr} + 2p_{\blbl}p_{\rr} - p_{\blbl}p_{\br}^{2}\right) \mw_{1,r}^{2} - 2p_{\blbl}p_{\br}p_{\rr}\mw_{1,r}^{3}\\ - p_{\blbl}p_{\rr}^{2}\mw_{1,r}^{4}. \nonumber
\end{multline}
Likewise, we deduce that
\begin{multline}
1 - \ml_{1,b} = p_{0} + p_{\rr} + \left(p_{\br}^{2} + 2p_{\blbl}p_{\br}\right) \left(1 - \ml_{1,r}\right) + \left(p_{\br}p_{\rr} + 2p_{\blbl}p_{\rr} - p_{\blbl}p_{\br}^{2}\right) \left(1 - \ml_{1,r}\right)^{2} \\- 2p_{\blbl}p_{\br}p_{\rr}\left(1 - \ml_{1,r}\right)^{3} - p_{\blbl}p_{\rr}^{2}\left(1 - \ml_{1,r}\right)^{4}.\nonumber
\end{multline}
Combining the above, we get
\begin{multline}\label{md_{1,b}_recursion_example_4}
\md_{1,b} = \md_{1,r}\Big[p_{\br}^{2} + 2p_{\blbl}p_{\br} + \left(p_{\br}p_{\rr} + 2p_{\blbl}p_{\rr} - p_{\blbl}p_{\br}^{2}\right)\left(1 - \ml_{1,r} + \mw_{1,r}\right) - 2p_{\blbl}p_{\br}p_{\rr}\\\left\{\left(1 - \ml_{1,r}\right)^{2} + \mw_{1,r}^{2} + \mw_{1,r}(1 - \ml_{1,r})\right\} - p_{\blbl}p_{\rr}^{2} \left(1 - \ml_{1,r} + \mw_{1,r}\right)\left\{\left(1 - \ml_{1,r}\right)^{2} + \mw_{1,r}^{2}\right\}\Big].
\end{multline}
We analyze the cases where $B = p_{\br}p_{\rr} + 2p_{\blbl}p_{\rr} - p_{\blbl}p_{\br}^{2}$ is non-negative and where $B$ is negative, separately. When $B \geqslant 0$, from \eqref{md_{1,b}_recursion_example_4} and using $p_{\rr} \leqslant 1 - p_{\blbl} - p_{\br}$, we have
\begin{multline}\label{B_nonnegative_example_4}
\md_{1,b} \leqslant \md_{1,r}\left(p_{\br}^{2} + 2p_{\blbl}p_{\br} + 2\left(p_{\br}p_{\rr} + 2p_{\blbl}p_{\rr} - p_{\blbl}p_{\br}^{2}\right)\right) \\
\leqslant \md_{1,r}\left(2p_{\br} + 4p_{\blbl} - p_{\br}^{2} - 4p_{\blbl}^{2} - 4p_{\br}p_{\blbl} - 2p_{\br}^{2}p_{\blbl}\right). 
\end{multline}
Defining $\mathcal{D}$ as above, and letting $f: \mathcal{D} \rightarrow \mathbb{R}$ be $f(x,y) = 2x + 4y - x^{2} - 4y^{2} - 4xy - 2x^{2}y$, we get the partial derivatives $\frac{\partial}{\partial x}f(x,y) = 2 - 2x - 4y - 4xy$ and $\frac{\partial}{\partial y}f(x,y) = 4 - 8y - 4x - 2x^{2}$. From the equation $\frac{\partial}{\partial y}f(x,y) = 0$, we have 
\begin{equation}\label{y_value_example_4_misere}
8y = 4 - 4x - 2x^{2} \implies y = \left(2 - 2x - x^{2}\right)/4.
\end{equation}
Substituting this in the equation $\frac{\partial}{\partial x}f(x,y) = 0$, we get the cubic polynomial equation $x^{3} + 3x^{2} - 2x = 0$, with roots $x_{1} = 0$, $x_{2} = -\frac{3}{2} - \frac{\sqrt{17}}{2}$ and $x_{3} = \frac{\sqrt{17}}{2} - \frac{3}{2}$, of which $x_{1}$ and $x_{3}$ lie in our domain of interest. The corresponding values of $y$, from \eqref{y_value_example_4_misere}, are $y_{1} = \frac{1}{2}$ and $y_{3} = \frac{\sqrt{17}}{8} - \frac{3}{8}$. The values of $f$ at these critical points are $f(x_{1}, y_{1}) = f\left(0, \frac{1}{2}\right) = 1$ and $f(x_{3}, y_{3}) = f\left(\frac{\sqrt{17}}{2} - \frac{3}{2}, \frac{\sqrt{17}}{8} - \frac{3}{8}\right) \approx 0.8866$. 

We examine the critical point $(x_{1}, y_{1})$ carefully. We have $p_{\br} = 0$ and $p_{\blbl} = \frac{1}{2}$. If $p_{\rr} < \frac{1}{2}$, then \eqref{B_nonnegative_example_4} shows that we end up with a strictly smaller value of $p_{\br}^{2} + 2p_{\blbl}p_{\br} + 2p_{\br}p_{\rr} + 4p_{\blbl}p_{\rr} - 2p_{\blbl}p_{\br}^{2}$. So we need only focus on the case where $p_{\rr} = \frac{1}{2}$ and $p_{0} = 0$. We then have, from \eqref{md_{1,b}_recursion_example_4}:
\begin{align}
\md_{1,b} &= \md_{1,r}\Big[2\left(\frac{1}{2}\right)^{2}\left(1 - \ml_{1,r} + \mw_{1,r}\right) - \left(\frac{1}{2}\right)^{3} \left(1 - \ml_{1,r} + \mw_{1,r}\right)\left\{\left(1 - \ml_{1,r}\right)^{2} + \mw_{1,r}^{2}\right\}\Big] \nonumber\\
&= \md_{1,r}\left(1 - \ml_{1,r} + \mw_{1,r}\right)\left[\frac{1}{2} - \frac{1}{8}\left\{\left(1 - \ml_{1,r}\right)^{2} + \mw_{1,r}^{2}\right\}\right] \nonumber
\end{align}
and $\left(1 - \ml_{1,r} + \mw_{1,r}\right)\left[\frac{1}{2} - \frac{1}{8}\left\{\left(1 - \ml_{1,r}\right)^{2} + \mw_{1,r}^{2}\right\}\right]$ equals $1$ iff $1 - \ml_{1,r} + \mw_{1,r} = 2$ and $\left(1 - \ml_{1,r}\right)^{2} + \mw_{1,r}^{2} = 0$, which cannot happen simultaneously. 
On the boundary of $\mathcal{D}$:
\begin{itemize}
\item When $y = 1-x$, we have $f(x,1-x) = 2x^{3} - 3x^{2} + 2x$, and $f'(x,1-x) = 6x^{2} - 6x + 2 = 6\left(x - \frac{1}{2}\right)^{2} + \frac{1}{2}$ is strictly positive for all $x$, so that $f(x,1-x)$ is strictly increasing on $[0,1]$ and the maximum is $f(1,0) = 1$.
\item When $x = 0$, we have $f(0,y) = 4y - 4y^{2} = 4y(1-y)$, so that $f'(0,y) = 4 - 8y = 8\left(\frac{1}{2} - y\right)$ is strictly positive for $y < \frac{1}{2}$. Thus $f(0,y)$ is strictly increasing on $\left[0, \frac{1}{2}\right)$ and strictly decreasing on $\left(\frac{1}{2}, 1\right]$, and the maximum is $f\left(0, \frac{1}{2}\right) = 1$. 
\item When $y = 0$, we have $f(x,0) = 2x - x^{2} = x(2 - x)$, so that $f'(x,0) = 2 - 2x = 2(1-x) > 0$ for $x < 1$. Thus, $f(x,0)$ is strictly increasing on $[0,1)$, with maximum $f(1,0) = 1$. 
\end{itemize}
When $B < 0$, from \eqref{md_{1,b}_recursion_example_4}, we have
\begin{multline}\label{B_negative_example_4}
\md_{1,b} \leqslant \md_{1,r}\left[p_{\br}^{2} + 2p_{\blbl}p_{\br}\right] \leqslant \md_{1,r}\left[p_{\br}^{2} + 2p_{\blbl}p_{\br} + p_{\blbl}^{2}\right] = \md_{1,r}\left(p_{\br} + p_{\blbl}\right)^{2} \leqslant \md_{1,r}, 
\end{multline}
and the only scenario under which equality holds in \eqref{B_negative_example_4} is if we have $p_{\br} = 1$. Thus, unless $p_{\br} = 1$, we have $\md_{1,b} < \md_{1,r}$ if we assume that $\md_{1,r}$ is strictly positive. A similar analysis yields that unless $q_{\br} = 1$, we have $\md_{1,r} < \md_{1,b}$ if $\md_{1,b}$ is assumed strictly positive. Clearly, these two inequalities cannot hold simultaneously. Therefore, we conclude that:
\begin{itemize}
\item if $p_{\br} = q_{\br} = 1$, then $\md_{1,b} = \md_{1,r} = 1$,
\item and in all other cases, $\md_{1,b} = \md_{1,r} = 0$.
\end{itemize}

For the escape game, using the generating functions above and recursions from \S\ref{subsec:escape_recursions}, we get
\begin{multline}\label{esl_{b}_esl_{r}_relationship_1}
\esl_{b} = \esl_{r}\Big\{\left(p_{\br}^{2} + 2p_{\br}p_{\rr}\right) + \left(p_{\br}^{2}p_{\blbl} + 4p_{\br}p_{\rr}p_{\blbl} + 4p_{\blbl}p_{\rr}^{2} - p_{\br}p_{\rr}\right) \esl_{r} -\\ \big\{2p_{\br}p_{\rr}p_{\blbl} + 4p_{\blbl}p_{\rr}^{2}\big\}\esl_{r}^{2} + p_{\blbl}p_{\rr}^{2} \esl_{r}^{3}\Big\}, 
\end{multline}
so that, if $C = p_{\br}^{2}p_{\blbl} + 4p_{\br}p_{\rr}p_{\blbl} + 4p_{\blbl}p_{\rr}^{2} - p_{\br}p_{\rr}$ is non-negative, using $p_{\rr} \leqslant 1 - p_{\br} - p_{\blbl}$,  
\begin{multline}
\esl_{b} \leqslant \esl_{r}\left(p_{\br}^{2} + p_{\br}p_{\rr} + p_{\br}^{2}p_{\blbl} + 4p_{\br}p_{\rr}p_{\blbl} + 5p_{\blbl}p_{\rr}^{2}\right) \\ \leqslant \esl_{r}\left(2 p_{\br}^{2}p_{\blbl} + 6p_{\br}p_{\blbl}^{2} - 7p_{\br}p_{\blbl} + p_{\br} + 5p_{\blbl}^{3} - 10p_{\blbl}^{2} + 5p_{\blbl}\right).\nonumber
\end{multline}
With the domain $\mathcal{D}$ as defined above, consider the function $f:\mathcal{D} \rightarrow \mathbb{R}$ with $f(x,y) = 2x^{2}y + 6xy^{2} - 7xy + x + 5y^{3} - 10y^{2} + 5y$. The partial derivatives of $f$ are given by $\frac{\partial}{\partial x}f(x,y) = 4xy + 6y^{2} - 7y + 1$ and $\frac{\partial}{\partial y}f(x,y) = 2x^{2} + 12xy - 7x + 15y^{2} - 20y + 5$. The equation $\frac{\partial}{\partial x}f(x,y) = 0$ gives
\begin{align}\label{x_value_escape_example_3}
4xy = 7y - 1 - 6y^{2} \implies x = \frac{7y - 1 - 6y^{2}}{4y},
\end{align}
and substituting this in the equation $\frac{\partial}{\partial y}f(x,y) = 0$, we get the quartic polynomial equation $24 y^{4} + 16y^{3} - 42y^{2} + 2 = 0$, with roots $y_{1} = 1$, $y_{2} \approx -1.6868$, $y_{3} \approx -0.21244$ and $y\approx 0.23255$, of which $y_{1}$ and $y_{2}$ are in our domain of interest. From \eqref{x_value_escape_example_3}, the corresponding values of $x$ are $x_{1} = 0$ and $x_{4} \approx 0.32613$, and the corresponding values of $f$ are $f(x_{1}, y_{1}) = 0$ and $f(x_{4}, y_{4}) \approx 0.635365$. On the boundary of the domain $\mathcal{D}$, we have:
\begin{itemize}
\item When $y = 1-x$, we have $f(x,1-x) = x^{2} + x^{2}(1-x) = 2x^{2} - x^{3}$, and $f'(x,1-x) = x(4 - 3x)$ is strictly positive for all $x \in [0,1]$, thus showing that $f(x,1-x)$ is strictly increasing and its maximum is $f(1,0) = 1$.
\item When $x = 0$, we have $f(0,y) = 5y^{3} - 10y^{2} + 5y = 5y(1-y)^{2}$, and $f'(0,y) = 15y^{2} - 20y + 5 = 5(1-y)(1-3y)$ is strictly positive for $y < \frac{1}{3}$. Therefore, $f(0,y)$ is strictly increasing on $\left[0,\frac{1}{3}\right)$ and strictly decreasing on $\left(\frac{1}{3}, 1\right]$, and its maximum is $f\left(0,\frac{1}{3}\right) = \frac{20}{27}$.
\item When $y = 0$, we have $f(x,0) = x$, with maximum $f(1,0) = 1$.
\end{itemize}
When $C$ is negative, from \eqref{esl_{b}_esl_{r}_relationship_1}, we get
\begin{multline}
\esl_{b} \leqslant \esl_{r}\Big\{\left(p_{\br}^{2} + 2p_{\br}p_{\rr}\right) + p_{\blbl}p_{\rr}^{2} \esl_{r}^{3}\Big\} \leqslant \esl_{r} \Big\{p_{\br}^{2} + 2p_{\br}p_{\rr} + p_{\blbl}p_{\rr}^{2}\Big\} \\ \leqslant \esl_{r} \left\{p_{\br}^{2} + 2p_{\br}p_{\rr} + p_{\rr}^{2}\right\} = \esl_{r}\left(p_{\br} + p_{\rr}\right)^{2}, \nonumber
\end{multline}
and the only way equality holds is if $p_{\br} = 1$. These observations tell us that unless $p_{\br} = 1$, we have $\esl_{b} < \esl_{r}$ if we assume that $\esl_{r}$ is strictly positive. Likewise, unless $q_{\br} = 1$, we have $\esl_{r} < \esl_{b}$ if we assume that $\esl_{b}$ is strictly positive. As these inequalities cannot hold simultaneously, we conclude that 
\begin{itemize}
\item when $p_{\br} = q_{\br} = 1$, we have $\esl_{b} = \esl_{r} = 1$,
\item and in all other cases, we have $\esl_{b} = \esl_{r} = 0$.
\end{itemize}

\subsection{Proof of Theorem~\ref{thm:main_example_2}}\label{subsec:main_example_2_proof} We prove the theorem only for normal games, as the argument goes through \emph{mutatis mutandis} for mis\`{e}re and escape games. For any vertex $v$, if $X_{v}$ is the number of blue children and $Y_{v}$ the number of red children, then conditioned on $\sigma(v) = b$, we have $X_{v} \sim \poi(\lambda p_{b})$ and $Y_{v} \sim \poi(\lambda p_{r})$, conditioned on $\sigma(v) = r$, we have $X_{v} \sim \poi(\lambda q_{b})$ and $Y_{v} \sim \poi(\lambda q_{r})$, and $X_{v}$ and $Y_{v}$ are always independent, all due to Poisson thinning.

The relevant generating functions in this example are
\begin{align}
& G_{b,\{b\}}(x_{b}) = \exp\left\{\lambda p_{b} (x_{b} - 1)\right\}, \quad G_{r,\{r\}}(x_{r}) = \exp\left\{\lambda q_{r} (x_{r} - 1)\right\};\nonumber\\
&G_{b,\{r\}}(x_{r}) = \exp\left\{\lambda p_{r} (x_{r}-1)\right\}, \quad G_{r,\{b\}}(x_{b}) = \exp\left\{\lambda q_{b} (x_{b} - 1)\right\}.\nonumber
\end{align}
These generating functions and the recursions of \S\ref{subsec:normal_recursions} yield $1 - \nw_{1,b} = f_{1} \circ f_{2}\left(1 - \nw_{1,b}\right)$ and $\nl_{1,b} = f_{1} \circ f_{2}\left(\nl_{1,b}\right)$, where $f_{1}(x) = \exp\left\{-\lambda p_{b} \exp\left\{-\lambda p_{r} x\right\}\right\}$ and $f_{2}(x) = \exp\left\{-\lambda q_{r} \exp\left\{-\lambda q_{b} x\right\}\right\}$. As $f_{1}$ and $f_{2}$ are strictly increasing, $1 - \nw_{1,b} = \max \FP\left(f_{1} \circ f_{2}\right)$ and $\nl_{1,b} = \min \FP\left(f_{1} \circ f_{2}\right)$. Therefore, $\nd_{1,b}$ is positive if and only if $f_{1} \circ f_{2}$ has at least two distinct fixed points in $[0,1]$. Let $x_{1}$ be a positive real with $f_{2}(x_{1}) = \frac{\ln(\lambda p_{b})}{2\lambda p_{r}}$, and let $0 < x_{2} < 1$ be \emph{any} constant. We now examine the signs of $f_{1} \circ f_{2}(x) - x$ at $x_{1}$ and $x_{2}$. Note that, given any $\epsilon > 0$, however small, we have
\begin{align}
\lim_{\lambda \rightarrow \infty} \lambda q_{r} e^{-\lambda q_{b} x_{2}} = 0 \implies f_{2}(x_{2}) = e^{-\lambda q_{r} e^{-\lambda q_{b} x_{2}}} \geqslant 1-\epsilon \text{ for all sufficiently large } \lambda.\nonumber
\end{align}
Therefore, we have 
\begin{align}
\lim_{\lambda \rightarrow \infty} f_{1} \circ f_{2}(x_{2}) = \exp\left\{-p_{b} \lim_{\lambda \rightarrow \infty} \left(\lambda e^{-\lambda p_{r} f_{2}(x_{2})}\right)\right\} \geqslant \exp\left\{-p_{b} \lim_{\lambda \rightarrow \infty} \left(\lambda e^{-\lambda p_{r} (1-\epsilon)}\right)\right\} = 1, \nonumber
\end{align}
so that for all $\lambda$ sufficiently large, we have $f_{1} \circ f_{2}(x_{2}) \geqslant x_{2}$ as $x_{2} < 1$. On the other hand,
\begin{align}
f_{2}(x_{1}) = \frac{\ln(\lambda p_{b})}{2\lambda p_{r}} \implies 
& x_{1} = \frac{\ln \lambda + \ln q_{r} - \ln\left(\ln 2 + \ln \lambda + \ln p_{r} - \ln\left(\ln \lambda + \ln p_{b}\right)\right)}{\lambda q_{b}},\nonumber
\end{align}
so that 
\begin{align}
f_{1}\circ f_{2}(x_{1}) - x_{1} 
&= \exp\left\{-\sqrt{\lambda p_{b}}\right\} - \frac{\ln \lambda}{\lambda q_{b}} - \frac{\ln q_{r}}{\lambda q_{b}} + \frac{\ln\left(\ln 2 + \ln \lambda + \ln p_{r} - \ln\left(\ln \lambda + \ln p_{b}\right)\right)}{\lambda q_{b}}.\nonumber
\end{align}
Noting that 
\begin{equation}
\frac{\ln\left(\ln 2 + \ln \lambda + \ln p_{r} - \ln\left(\ln \lambda + \ln p_{b}\right)\right)}{\lambda q_{b}} = o\left(\frac{\ln \lambda}{\lambda q_{b}}\right) \text{ and } \exp\left\{-\sqrt{\lambda p_{b}}\right\} = o\left(\frac{\ln \lambda}{\lambda}\right) \text{ as } \lambda \rightarrow \infty,\nonumber
\end{equation}
we conclude that $f_{1} \circ f_{2}(x_{1}) - x_{1} < 0$ for all sufficiently large $\lambda$. Finally, for every $\lambda > 0$, we have $f_{1} \circ f_{2}(0) > 0$ and $f_{1} \circ f_{2}(1) < 1$. Thus, for all sufficiently large $\lambda$, there is at least one root of $f_{1} \circ f_{2}(x) - x$ in the interval $(0, x_{1})$, at least one in $(x_{1}, x_{2})$, and a third in $(x_{2}, 1)$, guaranteeing that $\nd_{1,b} > 0$. Moreover, given \emph{any} $\epsilon > 0$, by choosing $x_{2} > 1 - \epsilon/2$ and then $\lambda$ sufficiently large so that $f_{1} \circ f_{2}(x_{2}) > x_{2}$ and $x_{1} < \epsilon/2$, we conclude that $\nl_{1,b} = \min \FP f_{1} \circ f_{2} < \epsilon/2$ and $1 - \nw_{1,b} = \max \FP f_{1} \circ f_{2} > 1 - \epsilon/2$, thus making $\nd_{1,b} > 1 - \epsilon$. Analogous conclusions hold for $\nd_{1,r}$, $\nd_{2,b}$ and $\nd_{2,r}$.

We now establish the remaining claims made in Theorem~\ref{thm:main_example_2}. If \eqref{poisson_second_cond} holds, then
\begin{align}
(f_{1} \circ f_{2}(x) - x)' 
&= \lambda^{4} p_{b} p_{r} q_{b} q_{r} \exp\left\{-\lambda p_{r} f_{2}(x) - \lambda p_{b} \exp\left\{-\lambda p_{r} f_{2}(x)\right\} - \lambda q_{b} x -\lambda q_{r} \exp\left\{-\lambda q_{b} x\right\}\right\} - 1 \nonumber\\
&\leqslant \lambda^{4} p_{b} p_{r} q_{b} q_{r} \exp\left\{-\lambda p_{r} e^{-\lambda q_{r}} - \lambda p_{b} \exp\left\{-\lambda p_{r} \exp\left\{-\lambda q_{r} e^{-\lambda q_{b}}\right\}\right\} - \lambda q_{r} e^{-\lambda q_{b}}\right\} - 1, \nonumber
\end{align} 
is strictly negative. Thus, $f_{1} \circ f_{2}(x) - x$ is strictly decreasing on $[0,1]$ and has precisely one root in $[0,1]$. Hence $\nd_{1,b} = 0$ in this case. In particular, since $p_{b}p_{r} = p_{b}(1-p_{b}) \leqslant \frac{1}{4}$, and likewise, $q_{b} q_{r} \leqslant \frac{1}{4}$, the draw probabilities are $0$ if we have $\lambda \leqslant 2$.

Setting $f_{3}(x) = e^{-\lambda p_{r} x}$ and $f_{4}(x) = e^{-\lambda q_{b} x}$, we have $(f_{1} \circ f_{2}(x))'' = \lambda^{5} p_{b} p_{r} q_{b}^{2} q_{r} f_{3}(f_{2}(x)) f_{1}(f_{2}(x)) f_{4}(x) f_{2}(x) \left[\lambda^{2} p_{r} q_{r} f_{4}(x) f_{2}(x) \left\{\lambda p_{b} f_{3}(f_{2}(x)) - 1\right\} + \left\{\lambda q_{r} f_{4}(x) - 1\right\}\right]$, so that the sign of $(f_{1} \circ f_{2}(x))''$ is the same as that of the function 
\begin{equation}
f_{5}(x) = \lambda^{2} p_{r} q_{r} f_{4}(x) f_{2}(x) \left\{\lambda p_{b} f_{3}(f_{2}(x)) - 1\right\} + \left\{\lambda q_{r} f_{4}(x) - 1\right\}. \nonumber
\end{equation}
Both $\lambda p_{b} f_{3}(f_{2}(x)) - 1$ and $\lambda q_{r} f_{4}(x) - 1$ are strictly decreasing in $x$ on $[0,1]$, so that the minima are attained at $x = 1$. If \eqref{normal_poisson_cond} holds, both $f_{1}$ and $f_{2}$ are strictly convex on $[0,1)$. Since $f_{1}$ and $f_{2}$ are also strictly increasing, $f_{1} \circ f_{2}$ is strictly convex on $[0,1)$. Since $f_{1} \circ f_{2}(0) > 0$ and $f_{1} \circ f_{2}(1) < 1$, there is precisely one fixed point of $f_{1} \circ f_{2}$ in $[0,1]$, and therefore, $\nd_{1,b} = 0$.

\section{Proof of Theorem~\ref{thm:main_2}}\label{sec:main_2_proof}
\subsection{Proof of Theorem~\ref{thm:main_2}, part \ref{main_2_part_1}}\label{subsec:main_2_part_1}
For every $j \in [m]$, we show that $\NW_{1,j} \subset \ESW_{j}$, $\NL_{2,j} \subseteq \EEL_{j}$, $\MW_{1,j} \subseteq \ESW_{j}$ and $\ML_{2,j} \subseteq \EEL_{j}$, which immediately gives us the desired conclusion. 

Fix a realization $T$ of $\mathcal{T}$. If $v \in \NW_{1,j}$, the normal game with $v$ as the initial vertex and P1 playing the first round is won by P1. Consider an escape game on $T(v)$ with $v$ as the initial vertex and Stopper playing the first round. Stopper employs the exact same strategy against Escaper as P1 does against P2 in the aforementioned normal game. If the normal game culminates in the token reaching a vertex $u$, at an odd distance from $v$, such that no child of $u$ has colour in $[m] \setminus S_{\sigma(u)}$, then the escape game also reaches $u$ in the corresponding round, and Escaper, whose turn it is to move, is unable to do so as there are no outgoing edges from $u$ that are permissible for her. Thus, Stopper wins the escape game, establishing that $v \in \ESW_{j}$ as well. Likewise, $\NL_{2,j} \subseteq \EEL_{j}$.

Now consider $v$ in $\MW_{1,j}$. The mis\`{e}re game with $v$ as the initial vertex and P1 playing the first round is won by P1. Consider an escape game on $T(v)$ with $v$ as the initial vertex and Stopper playing the first round. Stopper employs the exact same strategy against Escaper as P1 does against P2 in the aforementioned mis\`{e}re game. If the mis\`{e}re game culminates in the token reaching a vertex $u$, at an even distance from $v$, such that no child of $u$ has colour in $S_{\sigma(u)}$, then the escape game also reaches $u$ in the corresponding round, and Stopper, whose turn it is to move, is unable to do so since there are no outgoing edges from $u$ that are permissible for her. Thus, Stopper wins the escape game, proving that $v \in \ESW_{j}$. Likewise, $\ML_{2,j} \subseteq \EEL_{j}$.


\subsection{Proof of Theorem~\ref{thm:main_2}, part \ref{main_2_part_2}}\label{subsec:main_2_part_2} Given $x_{1}, \ldots, x_{r}$ in $[0,1]$ and $n_{1}, \ldots, n_{r} \in \mathbb{N}_{0}$ such that not every $n_{i}$ is $0$, and $x_{1} = \min\left\{x_{i}: i \in [r]\right\}$, we can write
\begin{multline}
1 - \prod_{t=1}^{r} (1 - x_{t})^{n_{t}} 
= 1 - \left(1 - x_{1}\right)^{\sum_{t=1}^{r} n_{t}} + \sum_{i=2}^{r} \left(1 - x_{1}\right)^{\sum_{t=1}^{i-1}n_{t}} \left\{\left(1-x_{1}\right)^{n_{i}} - \left(1 - x_{i}\right)^{n_{i}}\right\} \prod_{t=i+1}^{r} \left(1 - x_{t}\right)^{n_{t}} 
\\= x_{1} \sum_{i=0}^{\sum_{t=1}^{r} n_{t}-1} \left(1 - x_{1}\right)^{i} + \sum_{i=2}^{r} \left(1 - x_{1}\right)^{\sum_{t=1}^{i-1}n_{t}} \left(x_{i} - x_{1}\right) \left\{\sum_{j=0}^{n_{i}-1}\left(1-x_{1}\right)^{j}\left(1 - x_{i}\right)^{n_{i}-1-j}\right\} \prod_{t=i+1}^{r} \left(1 - x_{t}\right)^{n_{t}}
\\ \geqslant x_{1} + \sum_{i=2}^{r} (x_{i}-x_{1}) n_{i} \left(1 - x_{i}\right)^{n_{i}-1} \prod_{t \in [r] \setminus \{i\}} \left(1 - x_{t}\right)^{n_{t}}. \nonumber
\end{multline}
For $v$ to be in $\ESW_{j}$, either $v$ has no child of colour $k$ for any $k \in S_{j}$, or $v$ has at least one child $u$ in $\EEL_{k}$ for some $k \in S_{j}$. Letting $k_{0} \in S_{j}$ be such that $\eel_{k_{0}} = \min\left\{\eel_{k}: k \in S_{j}\right\}$, we get
\begin{align}
\esw_{j} &= \alpha_{j} + \sum_{\substack{n_{r} \in \mathbb{N}_{0}: r \in [m]\\\left(n_{k}: k \in S_{j}\right) \neq \mathbf{0}_{S_{j}}}}\left\{1 - \prod_{k \in S_{j}}\left(1 - \eel_{k}\right)^{n_{k}}\right\} \chi_{j}(n_{1}, \ldots, n_{m}) \nonumber\\
&= \alpha_{j} + \sum_{\mathbf{n}_{S_{j}} = \left(n_{k}: k \in S_{j}\right) \in \mathbb{N}_{0} \setminus \left\{\mathbf{0}_{S_{j}}\right\}}\Bigg\{\eel_{k_{0}} + \sum_{k \in S_{j} \setminus \{k_{0}\}} \left(\eel_{k} - \eel_{k_{0}}\right) n_{k} \left(1 - \eel_{k}\right)^{n_{k}-1} \nonumber\\&\prod_{t \in S_{j} \setminus \{k\}}\left(1 - \eel_{t}\right)^{n_{t}}\Bigg\} \chi_{j,S_{j}}\left(\mathbf{n}_{S_{j}}\right) \nonumber\\
&= \alpha_{j} + \eel_{k_{0}}\left(1 - \alpha_{j}\right) + \sum_{k \in S_{j}} \left(\eel_{k} - \eel_{k_{0}}\right) \sum_{\mathbf{n}_{S_{j}} = \left(n_{k}: k \in S_{j}\right) \in \mathbb{N}_{0} \setminus \left\{\mathbf{0}_{S_{j}}\right\}} n_{k} \left(1 - \eel_{k}\right)^{n_{k}-1} \nonumber\\& \prod_{t \in S_{j} \setminus \{k\}}\left(1 - \eel_{t}\right)^{n_{t}} \chi_{j,S_{j}}\left(\mathbf{n}_{S_{j}}\right) \nonumber\\
&= \alpha_{j} + \eel_{k_{0}}\left(1 - \alpha_{j}\right) + \sum_{k \in S_{j}} \left(\eel_{k} - \eel_{k_{0}}\right) \partial_{k} G_{j,S_{j}}\left(1 - \eel_{t}: t \in S_{j}\right). \nonumber
\end{align}
Observing that $\alpha_{j} + \eel_{k_{0}}\left(1 - \alpha_{j}\right) \geqslant \alpha_{j}\eel_{k_{0}} + \eel_{k_{0}}\left(1 - \alpha_{j}\right) = \eel_{k_{0}}$ gives the second inequality of \eqref{main_2_part_2_eq_1}. The remaining two claims of \ref{main_2_part_2} are established via very similar computations.

\subsection{Proof of Theorem~\ref{thm:main_2}, part \ref{main_2_part_3}}\label{subsec:main_2_part_3} Assume the hypothesis of \ref{main_2_part_3}, and $\bnl_{1}^{(n)} \preceq \bml_{1}^{(n)}$ for some $n \in \mathbb{N}$ (the base case, $n = 0$, is immediate as $\bml_{1}^{(0)} = \bnl_{1}^{(0)} = \mathbf{0}_{[m]}$). Then $\bml_{1}^{(n+2)} - \bnl_{1}^{(n+2)}$ equals
\begin{align}
F_{N}\left(\mathbf{1}_{[m]} - \bnl_{1}^{(n)}\right) - F_{M}\left(\mathbf{1}_{[m]} - \bml_{1}^{(n)}\right) \geqslant F_{N}\left(\mathbf{1}_{[m]} - \bnl_{1}^{(n)}\right) - F_{M}\left(\mathbf{1}_{[m]} - \bnl_{1}^{(n)}\right), \nonumber
\end{align}
where the inequality follows from the induction hypothesis and since $F_{M}$ is monotonically increasing.
\begin{multline}
\F_{N,j}\left(\mathbf{1}_{[m]} - \bnl_{1}^{(n)}\right) - F_{M,j}\left(\mathbf{1}_{[m]} - \bnl_{1}^{(n)}\right) = - G_{j, S_{j}}\left(1 - G_{k, [m] \setminus S_{k}}\left(1 - \nl_{1,\ell}^{(n)}: \ell \in [m] \setminus S_{k}\right): k \in S_{j}\right)\\+ G_{j,S_{j}}\left(1 - G_{k,[m] \setminus S_{k}}\left(1 - \nl_{1,\ell}^{(n)}: \ell \in [m] \setminus S_{k}\right) + \beta_{k}: k \in S_{j}\right) - \alpha_{j} \geqslant  - \alpha_{j} + \\ \sum_{i \in S_{j}} \beta_{i} \partial_{i} G_{j,S_{j}}\left(1 - G_{k,[m] \setminus S_{k}}\left(1 - \nl_{1,\ell}^{(n)}: \ell \in [m] \setminus S_{k}\right): k \in S_{j}\right) = \sum_{i \in S_{j}} \beta_{i} \partial_{i} G_{j,S_{j}}\left(\nw_{2,k}^{(n+1)}: k \in S_{j}\right) - \alpha_{j}, \nonumber
\end{multline}
where the inequality holds as $G_{j,S_{j}}$ is convex (see, for example, [\cite{simon_blume}, Theorem 21.3]) and the last equality follows from \eqref{normal_recur_2}. Given the continuously differentiable pgf $G: [0,1]^{r} \rightarrow [0,1]$ corresponding to a probability distribution $\chi$ supported on $\mathbb{N}_{0}^{r}$ for $r \in \mathbb{N}$, we have, for $i \in [r]$, 
\begin{align}
\partial_{i} G(x_{1}, \ldots, x_{r}) &= \sum_{n_{k} \in \mathbb{N}_{0}: k \in [r]} \left(\prod_{k \in [r] \setminus \{i\}} x_{k}^{n_{k}}\right) n_{i} x_{i}^{n_{i}-1} \chi(n_{1}, \ldots, n_{r}), \nonumber
\end{align}
which is a power series with non-negative coefficients. If $(x_{1}, \ldots, x_{r})$ and $(y_{1}, \ldots, y_{r})$ in $[0,1]^{r}$ satisfy $(x_{1}, \ldots, x_{r}) \preceq (y_{1}, \ldots, y_{r})$, then $\partial_{i} G(x_{1}, \ldots, x_{r}) \leqslant \partial_{i} G(y_{1}, \ldots, y_{r})$. Since $0 = \nw_{2,k}^{(1)} \leqslant \nw_{2,k}^{(n+1)}$ for all $n \in \mathbb{N}_{0}$ and all $k \in S_{j}$, hence $\partial_{i} G_{j,S_{j}}\left(\mathbf{0}_{S_{j}}\right) \leqslant \partial_{i} G_{j,S_{j}}\left(\nw_{2,k}^{(n+1)}: k \in S_{j}\right)$. Moreover, $\partial_{i} G_{j,S_{j}}\left(\mathbf{0}_{S_{j}}\right) = \Prob\left[X_{v,i} = 1, X_{v,k} = 0 \text{ for all } k \in S_{j} \setminus \{i\}\big|\sigma(v) = j\right]$. Thus, if
\eqref{main_2_part_3_cond_1} holds, we have $\bnl_{1}^{(n+2)} \preceq \bml_{1}^{(n+2)}$, thus completing the inductive proof. Taking limits as $n \rightarrow \infty$ gives the rest. 

The second part of \ref{main_2_part_3} is also proved via induction. Assuming $\bnw_{1}^{(n)} \prec \bmw_{1}^{(n)}$, for each $j \in [m]$,
\begin{align}
\mw_{1,j}^{(n+2)} - \nw_{1,j}^{(n+2)} &= F_{M,j}\left(\bmw_{1}^{(n)}\right) - F_{N,j}\left(\bnw_{1}^{(n)}\right) \geqslant F_{M,j}\left(\bmw_{1}^{(n)}\right) - F_{N,j}\left(\bmw_{1}^{(n)}\right) \nonumber\\
&= \alpha_{j} -  G_{j,S_{j}}\left(1 - G_{k,[m] \setminus S_{k}}\left(\mw_{1,\ell}^{(n)}: \ell \in [m] \setminus S_{k}\right) + \beta_{k}: k \in S_{j}\right) \nonumber\\&+ G_{j,S_{j}}\left(1 - G_{k,[m] \setminus S_{k}}\left(\mw_{1,\ell}^{(n)}: \ell \in [m] \setminus S_{k}\right): k \in S_{j}\right) \nonumber\\
&\geqslant \alpha_{j} - \sum_{i \in S_{j}} \beta_{i} \partial_{i} G_{j,S_{j}}\left(1 - G_{k,[m] \setminus S_{k}}\left(\mw_{1,\ell}^{(n)}: \ell \in [m] \setminus S_{k}\right) + \beta_{k}: k \in S_{j}\right) \nonumber\\
&= \alpha_{j} - \sum_{i \in S_{j}} \beta_{i} \partial_{i} G_{j,S_{j}}\left(1 - \ml_{2,k}^{(n+1)}: k \in S_{j}\right). \nonumber
\end{align}
From $0 = \ml_{2,k}^{(1)} \leqslant \ml_{2,k}^{(n+1)}$ for all $n \in \mathbb{N}_{0}$ and the monotonically increasing $\partial_{i} G_{j,S_{j}}$, it is enough if we have $\alpha_{j} \geqslant \sum_{i \in S_{j}} \beta_{i} \partial_{i} G_{j,S_{j}}\left(\mathbf{1}_{S_{j}}\right)$, which is equivalent to \eqref{main_2_part_3_cond_2}.


\section{Proof of Theorem~\ref{thm:main_3}}\label{sec:main_3_proof}
The proof of Theorem~\ref{thm:main_3} happens via several lemmas, as follows. Recall the law $\mathcal{L}_{\bm{\chi}}$ of $\mathcal{T}_{[m], \mathbf{p}, \bm{\chi}}$ defined in \S\ref{subsec:main_results}, and the metric $d_{0}$ from \eqref{metric_d_{0}}.
\begin{lemma}\label{main_3_lem_1}
The probabilities $\alpha_{j,\bm{\chi}}$ and $\beta_{j,\bm{\chi}}$ are continuous functions of $\mathcal{L}_{\bm{\chi}}$ with respect to $d_{0}$.
\end{lemma}
\begin{proof}
We show the proof only for $\alpha_{j, \bm{\chi}}$. Given laws $\mathcal{L}_{\bm{\chi}}$ and $\mathcal{L}_{\bm{\eta}}$, we have
\begin{align}
\left|\alpha_{j,\bm{\chi}} - \alpha_{j,\bm{\eta}}\right| &= \left|\chi_{j,S_{j}}\left(\mathbf{0}_{S_{j}}\right) - \eta_{j,S_{j}}\left(\mathbf{0}_{S_{j}}\right)\right| \nonumber\\
&\leqslant \sum_{\mathbf{n}_{[m] \setminus S_{j}} = \left(n_{k}: k \in [m] \setminus S_{j}\right) \in \mathbb{N}_{0}^{|[m] \setminus S_{j}|}} \left|\chi_{j}\left(\mathbf{0}_{S_{j}} \vee \mathbf{n}_{[m] \setminus S_{j}}\right) - \eta_{j}\left(\mathbf{0}_{S_{j}} \vee \mathbf{n}_{[m] \setminus S_{j}}\right)\right| \leqslant 2d_{0}\left(\mathcal{L}_{\bm{\chi}}, \mathcal{L}_{\bm{\eta}}\right).\nonumber \qedhere
\end{align}
\end{proof}

\begin{lemma}\label{main_3_lem_2}
Fix $j \in [m]$, $S \subset [m]$ and $\alpha \in (0,1)$. Set $S_{\alpha} = \sum_{n \in \mathbb{N}} n \alpha^{n-1}$. Fix $\mathbf{x}_{S} = \left(x_{k}: k \in S\right)$ and $\mathbf{y}_{S} = \left(y_{k}: k \in S\right)$ in $[0,\alpha]^{S}$. Then $\mathcal{L}_{\bm{\chi}}$ satisfies the following inequality:
\begin{equation}\label{main_3_lem_2_eq_1}
\left|G_{j,S,\bm{\chi}}\left(\mathbf{x}_{S}\right) - G_{j,S,\bm{\chi}}\left(\mathbf{y}_{S}\right)\right| \leqslant \max_{k \in S} \left|x_{k} - y_{k}\right| S_{\alpha}.
\end{equation}
If $\mathcal{E}_{j,S,\bm{\chi}} = \E_{\mathcal{L}_{\bm{\chi}}}\left[\sum_{k \in S} X_{v,k}\big|\sigma(v) = j\right]$ is finite, for all $\mathbf{x}_{S}$ and $\mathbf{y}_{S}$ in $[0,1]^{S}$, we have
\begin{equation}\label{main_3_lem_2_eq_2}
\left|G_{j,S,\bm{\chi}}\left(\mathbf{x}_{S}\right) - G_{j,S,\bm{\chi}}\left(\mathbf{y}_{S}\right)\right| \leqslant \max_{k \in S} \left|x_{k} - y_{k}\right|\mathcal{E}_{j,S,\bm{\chi}}.
\end{equation}  
\end{lemma}
\begin{proof}
Given tuples $(u_{1}, \ldots, u_{r})$ and $(v_{1}, \ldots, v_{r})$ in $[0,\alpha]^{r}$ and $(n_{1}, \ldots, n_{r}) \in \mathbb{N}_{0}$, we can write
\begin{align}
\left|\prod_{k \in [r]} u_{k}^{n_{k}} - \prod_{k \in [r]} v_{k}^{n_{k}}\right| &\leqslant \sum_{i \in [r]} \left(\prod_{j=1}^{i-1} u_{j}^{n_{j}}\right) \left|u_{i}^{n_{i}} - v_{i}^{n_{i}}\right| \left(\prod_{j=i+1}^{r} v_{j}^{n_{j}}\right)\nonumber\\
&= \sum_{i \in [r]} \left(\prod_{j=1}^{i-1} u_{j}^{n_{j}}\right) \left|u_{i}-v_{i}\right| \left(\sum_{t=0}^{n_{i}-1} u_{i}^{t} v_{i}^{n_{i}-1-t}\right) \left(\prod_{j=i+1}^{r} v_{j}^{n_{j}}\right)\nonumber\\
&\leqslant \sum_{i \in [r]} \alpha^{\sum_{j=1}^{i-1} n_{j}} \left|u_{i}-v_{i}\right| n_{i} \alpha^{n_{i}-1} \alpha^{\sum_{j=i+1}^{r} n_{j}} \leqslant \max_{k \in [r]} |u_{k} - v_{k}| \left(\sum_{i \in [r]} n_{i}\right) \alpha^{\sum_{j \in [r]} n_{j}-1}.\nonumber
\end{align} 
This allows us to bound $\left|G_{j,S,\bm{\chi}}\left(\mathbf{x}_{S}\right) - G_{j,S,\bm{\chi}}\left(\mathbf{y}_{S}\right)\right|$ by
\begin{align}
&\sum_{\mathbf{n}_{S} = (n_{k}: k \in S) \in \mathbb{N}_{0}^{S}} \left|\prod_{k \in S} x_{k}^{n_{k}} - \prod_{k \in S} y_{k}^{n_{k}}\right| \chi_{j,S}\left(\mathbf{n}_{S}\right) \leqslant \sum_{\mathbf{n}_{S} = (n_{k}: k \in S) \in \mathbb{N}_{0}^{S}} \max_{k \in S}|x_{k} - y_{k}| \left(\sum_{k \in S} n_{k}\right) \alpha^{\sum_{k \in S} n_{k} - 1}\chi_{j,S}\left(\mathbf{n}_{S}\right) \nonumber\\
&= \max_{k \in S}|x_{k} - y_{k}| \sum_{M \in \mathbb{N}_{0}} M \alpha^{M-1} \sum_{\mathbf{n}_{S} = (n_{k}: k \in S) \in \mathbb{N}_{0}^{S}: \sum_{k \in S} n_{k} = M} \chi_{j,S}\left(\mathbf{n}_{S}\right) \nonumber\\
&= \max_{k \in S}|x_{k} - y_{k}| \sum_{M \in \mathbb{N}} M \alpha^{M-1} \Prob_{\mathcal{L}_{\bm{\chi}}}\left[\sum_{k \in S} X_{\phi,k} = M\big|\sigma(\phi) = j\right] \leqslant \max_{k \in S}|x_{k} - y_{k}| S_{\alpha},\nonumber
\end{align}
which gives us \eqref{main_3_lem_2_eq_1}. To deduce \eqref{main_3_lem_2_eq_2}, similar computations lead to the general inequality
\begin{equation}
\left|\prod_{k \in [r]} u_{k}^{n_{k}} - \prod_{k \in [r]} v_{k}^{n_{k}}\right| \leqslant \max_{k \in [r]} |u_{k} - v_{k}| \left(\sum_{i \in [r]} n_{i}\right), \nonumber
\end{equation}
and using this, we bound $\left|G_{j,S,\bm{\chi}}\left(\mathbf{x}_{S}\right) - G_{j,S,\bm{\chi}}\left(\mathbf{y}_{S}\right)\right|$ by
\begin{multline}
\sum_{\mathbf{n}_{S} = (n_{k}: k \in S) \in \mathbb{N}_{0}^{S}} \max_{k \in S}|x_{k} - y_{k}| \left(\sum_{k \in S} n_{k}\right)\chi_{j,S}\left(\mathbf{n}_{S}\right) = \max_{k \in S}|x_{k} - y_{k}| \sum_{M \in \mathbb{N}_{0}} M \sum_{\substack{n_{k} \in \mathbb{N}_{0}: k \in S\\ \sum_{k \in S} n_{k} = M}} \chi_{j,S}\left(\mathbf{n}_{S}\right) \\= \max_{k \in S}|x_{k} - y_{k}| \sum_{M \in \mathbb{N}_{0}} M \Prob_{\mathcal{L}_{\bm{\chi}}}\left[\sum_{k \in S} X_{\phi,k} = M\big|\sigma(\phi) = j\right] = \max_{k \in S}|x_{k} - y_{k}| \mathcal{E}_{j,S,\bm{\chi}}.\nonumber \qedhere
\end{multline}
\end{proof}

\begin{lemma}\label{main_3_lem_3}
Given laws $\mathcal{L}_{\bm{\chi}}$ and $\mathcal{L}_{\bm{\eta}}$, subset $S \subset [m]$ and $\mathbf{x}_{S} = (x_{k}: k \in S) \in [0,1]^{S}$, we have
\begin{equation}
\left|G_{j,S,\bm{\chi}}\left(\mathbf{x}_{S}\right) - G_{j,S,\bm{\eta}}\left(\mathbf{x}_{S}\right)\right| \leqslant 2d_{0}\left(\mathcal{L}_{\bm{\chi}}, \mathcal{L}_{\bm{\eta}}\right).
\end{equation}
\end{lemma}
\begin{proof}
$\begin{aligned}[t]
&\left|G_{j,S,\bm{\chi}}\left(\mathbf{x}_{S}\right) - G_{j,S,\bm{\eta}}\left(\mathbf{x}_{S}\right)\right| = \left|\sum_{\mathbf{n}_{S} = (n_{k}: k \in S) \in \mathbb{N}_{0}^{S}}\prod_{k \in S} x_{k}^{n_{k}} \chi_{j,S}(\mathbf{n}_{S}) - \sum_{\mathbf{n}_{S} = (n_{k}: k \in S) \in \mathbb{N}_{0}^{S}}\prod_{k \in S} x_{k}^{n_{k}} \eta_{j,S}(\mathbf{n}_{S})\right| \nonumber\\
&\leqslant \sum_{\mathbf{n}_{S} = (n_{k}: k \in S) \in \mathbb{N}_{0}^{S}}\prod_{k \in S} x_{k}^{n_{k}} \left|\chi_{j,S}(\mathbf{n}_{S}) - \eta_{j,S}(\mathbf{n}_{S})\right| \nonumber\\
&\leqslant \sum_{\mathbf{n}_{[m] \setminus S} = (n_{k}: k \in [m] \setminus S) \in \mathbb{N}_{0}^{[m] \setminus S}} \sum_{\mathbf{n}_{S} = (n_{k}: k \in S) \in \mathbb{N}_{0}^{S}} \left|\chi_{j}\left(\mathbf{n}_{S} \vee \mathbf{n}_{[m] \setminus S}\right) - \eta_{j}\left(\mathbf{n}_{S} \vee \mathbf{n}_{[m] \setminus S}\right)\right| = 2||\chi_{j} - \eta_{j}||_{\tv}. \nonumber \qedhere
\end{aligned}$ 
\end{proof}

\begin{lemma}\label{main_3_lem_5}
Given laws $\mathcal{L}_{\bm{\chi}}$ and $\mathcal{L}_{\bm{\eta}}$ with $d_{0}\left(\mathcal{L}_{\bm{\chi}}, \mathcal{L}_{\bm{\eta}}\right) \leqslant \epsilon_{1}$, and $\mathbf{x} = (x_{k}: k \in [m])$ and $\mathbf{y} = (y_{k}: k \in [m])$ in $[0,\alpha]^{m}$, for some $\alpha \in (0,1)$, with $\max_{k \in [m]}|x_{k} - y_{k}| \leqslant \epsilon_{2}$, we have
\begin{equation}\label{main_3_lem_5_eq_1}
\left|G_{j,S,\bm{\chi}}\left(x_{k}: k \in S\right) - G_{j,S,\bm{\eta}}\left(y_{k}: k \in S\right)\right| \leqslant 2\epsilon_{1} + S_{\alpha} \epsilon_{2}.
\end{equation}
for any subset $S$ of $[m]$. If $\mathcal{E}_{j,S,\bm{\chi}} = \E_{\mathcal{L}_{\bm{\chi}}}\left[\sum_{k \in S} X_{v,k}\big|\sigma(v) = j\right]$ is finite, for $\mathbf{x} = (x_{k}: k \in [m])$ and $\mathbf{y} = (y_{k}: k \in [m])$ in $[0,1]^{m}$ with $\max_{k \in [m]}|x_{k} - y_{k}| \leqslant \epsilon_{2}$, we have
\begin{equation}\label{main_3_lem_5_eq_2}
\left|G_{j,S,\bm{\chi}}\left(x_{k}: k \in S\right) - G_{j,S,\bm{\eta}}\left(y_{k}: k \in S\right)\right| \leqslant 2\epsilon_{1} + \mathcal{E}_{j,S,\bm{\chi}} \epsilon_{2}.
\end{equation}
\end{lemma}
\begin{proof} 
The left sides of both \eqref{main_3_lem_5_eq_1} and \eqref{main_3_lem_5_eq_2} can be bounded above as follows:
\begin{multline}
\left|G_{j,S,\bm{\chi}}\left(x_{k}: k \in S\right) - G_{j,S,\bm{\eta}}\left(y_{k}: k \in S\right)\right| \leqslant \left|G_{j,S,\bm{\chi}}\left(x_{k}: k \in S\right) - G_{j,S,\bm{\chi}}\left(y_{k}: k \in S\right)\right| \\+ \left|G_{j,S,\bm{\chi}}\left(y_{k}: k \in S\right) - G_{j,S,\bm{\eta}}\left(y_{k}: k \in S\right)\right|.\nonumber
\end{multline}
Under the hypothesis of the first part of Lemma \ref{main_3_lem_5}, we use \eqref{main_3_lem_2_eq_1} to get
\begin{equation}
\left|G_{j,S,\bm{\chi}}\left(x_{k}: k \in S\right) - G_{j,S,\bm{\chi}}\left(y_{k}: k \in S\right)\right| \leqslant \max_{k \in S}|x_{k} - y_{k}| S_{\alpha} \leqslant \epsilon_{2} S_{\alpha},\nonumber
\end{equation}
and by Lemma~\ref{main_3_lem_3}, we have $\left|G_{j,S,\bm{\chi}}\left(y_{k}: k \in S\right) - G_{j,S,\bm{\eta}}\left(y_{k}: k \in S\right)\right| \leqslant 2d_{0}\left(\mathcal{L}_{\bm{\chi}}, \mathcal{L}_{\bm{\eta}}\right) \leqslant 2\epsilon_{1}$. Combining everything, we get \eqref{main_3_lem_5_eq_1}. To deduce \eqref{main_3_lem_5_eq_2}, we apply \eqref{main_3_lem_2_eq_2} instead of \eqref{main_3_lem_2_eq_1}.
\end{proof} 

\begin{lemma}\label{main_3_lem_4}
For each $j \in [m]$ and $n \in \mathbb{N}$, we have $\nw_{1,j,\bm{\chi}}^{(n)} \leqslant 1 - \nl_{1,j,\bm{\chi}}^{(n)} \leqslant 1 - \alpha_{j,\bm{\chi}}$, whereas $\nw_{2,j,\bm{\chi}}^{(n)} \leqslant 1 - \nl_{2,j,\bm{\chi}}^{(n)} \leqslant 1 - \beta_{j,\bm{\chi}}$. We also have $\mw_{1,j,\bm{\chi}}^{(n)} \leqslant 1 - \ml_{1,j,\bm{\chi}}^{(n)} \leqslant 1 + \alpha_{j,\bm{\chi}} - G_{j,S_{j},\bm{\chi}}\left(\beta_{k,\bm{\chi}}: k \in S_{j}\right)$, whereas $\mw_{2,j,\bm{\chi}}^{(n)} \leqslant 1 - \ml_{2,j,\bm{\chi}}^{(n)} \leqslant 1 + \beta_{j,\bm{\chi}} - G_{j,[m] \setminus S_{j},\bm{\chi}}\left(\alpha_{k,\bm{\chi}}: k \in [m] \setminus S_{j}\right)$. 
\end{lemma}
\begin{proof}
If the initial vertex $v$, with $\sigma(v) = j$, has no child of colour $k$ for any $k \in S_{j}$, and P1 plays the first round, she loses. Thus $\nl_{1,j}^{(n)} \geqslant \alpha_{j,\bm{\chi}}$ for each $n \in \mathbb{N}$. This yields $\nw_{1,j,\bm{\chi}}^{(n)} \leqslant 1 - \nl_{1,j,\bm{\chi}}^{(n)} \leqslant 1 - \alpha_{j,\bm{\chi}}$, as desired. The claim about $\nw_{2,j,\bm{\chi}}^{(n)}$ and $1 - \nl_{2,j,\bm{\chi}}^{(n)}$ follows similarly. 

Let $v$, with $\sigma(v) = j$, be the initial vertex. If for every child $u$ of $v$ that is of colour $k$ for any $k \in S_{j}$, $u$ has no child of colour $\ell$ for any $\ell \in [m] \setminus S_{k}$, and P1 plays the first round of the mis\`{e}re game, she is forced to move the token to such a $u$, and P2 fails to move in the second round, thus winning the game. Thus, for each $n \in \mathbb{N}$,
\begin{align}
\ml_{1,j,\bm{\chi}}^{(n)} &\geqslant \sum_{(n_{k}: k \in S_{j}) \in \mathbb{N}_{0}^{|S_{j}|} \setminus \left\{\mathbf{0}_{S_{j}}\right\}} \prod_{k \in S_{j}} \beta_{k,\bm{\chi}}^{n_{k}} \chi_{j,S_{j}}\left(n_{k}: k \in S_{j}\right) = G_{j,S_{j},\bm{\chi}}\left(\beta_{k,\bm{\chi}}: k \in S_{j}\right) - \alpha_{j,\bm{\chi}}.\nonumber
\end{align}
Consequently, both $\mw_{1,j,\bm{\chi}}^{(n)}$ and $1 - \ml_{1,j,\bm{\chi}}^{(n)}$ are bounded above by $1 + \alpha_{j,\bm{\chi}} - G_{j,S_{j},\bm{\chi}}\left(\beta_{k,\bm{\chi}}: k \in S_{j}\right)$. The claim about $\mw_{2,j,\bm{\chi}}^{(n)}$ and $1 - \ml_{2,j,\bm{\chi}}^{(n)}$ follows similarly.
\end{proof}

\begin{proof}[Proof of Theorem~\ref{thm:main_3}]
Proof of \ref{main_3_part_1}: We establish the claim for $\nw_{1,j,\bm{\chi}}$, for any fixed $j \in [m]$. Lemma~\ref{lem:normal_compactness_consequence} implies that $\nw_{1,j,\bm{\chi}}$ is the limit of the increasing sequence $\left\{\nw_{1,j,\bm{\chi}}^{(n)}\right\}_{n}$. It thus suffices to show that $\nw_{1,j,\bm{\chi}}^{(n)}$ is a continuous function of $\mathcal{L}_{\bm{\chi}}$ for each $n \in \mathbb{N}_{0}$, with respect to $d_{0}$. 

We assume, for some $n \in \mathbb{N}_{0}$, that $\nw_{1,j,\bm{\chi}}^{(n)}$ is continuous in $\mathcal{L}_{\bm{\chi}}$, and show that $\nw_{1,j,\bm{\chi}}^{(n+2)} = F_{N,j,\bm{\chi}}\left(\bnw_{1,\bm{\chi}}^{(n)}\right)$ is continuous in $\mathcal{L}_{\bm{\chi}}$ as well. The base case of $n = 0$ is immediate since $\nw_{1,j,\bm{\chi}}^{(0)} = 0$. 

First, consider $\mathcal{L}_{\bm{\chi}}$ in $\mathcal{D}_{1}$. From Lemma~\ref{main_3_lem_1} and Lemma~\ref{main_3_lem_4}, if we choose $\mathcal{L}_{\bm{\eta}}$ such that 
\begin{equation}\label{dist_laws_main_3_part_1}
d_{0}\left(\mathcal{L}_{\bm{\eta}}, \mathcal{L}_{\bm{\chi}}\right) \leqslant \epsilon_{1} < \frac{1}{4} \min\left\{\alpha_{j,\bm{\chi}}: j \in [m]\right\}, 
\end{equation}
then $\max\left\{\nw_{1,j,\bm{\chi}}^{(n)}, \nw_{1,j,\bm{\eta}}^{(n)}\right\} \leqslant 1 - \frac{1}{2} \alpha_{j,\bm{\chi}} \text{ for each } j \in [m]$. Setting $\alpha = \max\left\{1 - \frac{\alpha_{j,\bm{\chi}}}{2}: j \in [m]\right\}$, we use \eqref{main_3_lem_5_eq_1} of Lemma~\ref{main_3_lem_5} and \eqref{dist_laws_main_3_part_1} to write, for each $k \in [m]$,
\begin{multline}\label{main_3_part_1_upper_bound_1}
\left|G_{k,[m] \setminus S_{k},\bm{\chi}}\left(\nw_{1,\ell,\bm{\chi}}^{(n)}: \ell \in [m] \setminus S_{k}\right) - G_{k,[m] \setminus S_{k},\bm{\eta}}\left(\nw_{1,\ell,\bm{\eta}}^{(n)}: \ell \in [m] \setminus S_{k}\right)\right|\\ \leqslant 2\epsilon_{1} + S_{\alpha}\max_{j \in [m]}\left|\nw_{1,j,\bm{\chi}}^{(n)} - \nw_{1,j,\bm{\eta}}^{(n)}\right|. 
\end{multline}
If $\mathcal{L}_{\bm{\chi}} \in \mathcal{D}_{4}$, then for $\mathcal{L}_{\bm{\eta}}$ satisfying \eqref{dist_laws_main_3_part_1}, using \eqref{main_3_lem_5_eq_2} of Lemma~\ref{main_3_lem_5}, we have, for each $k \in [m]$,
\begin{multline}\label{main_3_part_1_upper_bound_2}
\left|G_{k,[m] \setminus S_{k},\bm{\chi}}\left(\nw_{1,\ell,\bm{\chi}}^{(n)}: \ell \in [m] \setminus S_{k}\right) - G_{k,[m] \setminus S_{k},\bm{\eta}}\left(\nw_{1,\ell,\bm{\eta}}^{(n)}: \ell \in [m] \setminus S_{k}\right)\right|\\ \leqslant 2\epsilon_{1} + \mathcal{E}_{k,[m] \setminus S_{k},\bm{\chi}}\max_{j \in [m]}\left|\nw_{1,j,\bm{\chi}}^{(n)} - \nw_{1,j,\bm{\eta}}^{(n)}\right|. 
\end{multline}
We let $\epsilon_{3}$ denote the upper bound in \eqref{main_3_part_1_upper_bound_1} when $\mathcal{L}_{\bm{\chi}} \in \mathcal{D}_{1}$, and we let 
\begin{equation}\label{epsilon_{3}_defn}
\epsilon_{3} = 2\epsilon_{1} + \mathcal{E}_{\bm{\chi}}\max_{j \in [m]}\left|\nw_{1,j,\bm{\chi}}^{(n)} - \nw_{1,j,\bm{\eta}}^{(n)}\right|, \text{ with } \mathcal{E}_{\bm{\chi}} = \max_{k \in [m]}\mathcal{E}_{k,[m] \setminus S_{k},\bm{\chi}},
\end{equation}
when $\mathcal{L}_{\bm{\chi}} \in \mathcal{D}_{4}$. Next, from Lemma~\ref{main_3_lem_4} and \eqref{normal_recur_4}, we have $1 - G_{k,[m] \setminus S_{k},\bm{\chi}}\left(\nw_{1,\ell,\bm{\chi}}^{(n)}: \ell \in [m] \setminus S_{k}\right) = 1 - \nl_{2,k,\bm{\chi}}^{(n+1)} \leqslant 1 - \beta_{k,\bm{\chi}}$. If $\mathcal{L}_{\bm{\chi}}$ in $\mathcal{D}_{2}$ and $\mathcal{L}_{\bm{\eta}}$ satisfies
\begin{equation}\label{dist_cond_main_3_part_1}
d_{0}\left(\mathcal{L}_{\bm{\eta}}, \mathcal{L}_{\bm{\chi}}\right) \leqslant \epsilon_{1} < \frac{1}{4} \min\left\{\beta_{j,\bm{\chi}}: j \in [m]\right\},
\end{equation}
then from Lemma~\ref{main_3_lem_1} and Lemma~\ref{main_3_lem_4}, we have, for each $j \in [m]$,
\begin{equation}
\max\left\{1 - G_{k,[m] \setminus S_{k},\bm{\chi}}\left(\nw_{1,\ell,\bm{\chi}}^{(n)}: \ell \in [m] \setminus S_{k}\right), 1 - G_{k,[m] \setminus S_{k},\bm{\eta}}\left(\nw_{1,\ell,\bm{\eta}}^{(n)}: \ell \in [m] \setminus S_{k}\right)\right\} \leqslant 1 - \frac{1}{2} \beta_{j,\bm{\chi}}.\nonumber
\end{equation}
Setting $\beta = \max\left\{1 - \frac{\beta_{j,\bm{\chi}}}{2}: j \in [m]\right\}$, we use \eqref{main_3_lem_5_eq_1} of Lemma~\ref{main_3_lem_5} and \eqref{dist_cond_main_3_part_1} to write
\begin{multline}\label{main_3_part_1_upper_bound_3}
\left|F_{N,j,\bm{\chi}}\left(\bnw_{1,\bm{\chi}}^{(n)}\right) - F_{N,j,\bm{\eta}}\left(\bnw_{1,\bm{\eta}}^{(n)}\right)\right| \leqslant 2\epsilon_{1} + S_{\beta} \max_{k \in [m]}\Big|G_{k,[m] \setminus S_{k},\bm{\chi}}\left(\nw_{1,\ell,\bm{\chi}}^{(n)}: \ell \in [m] \setminus S_{k}\right) \\- G_{k,[m] \setminus S_{k},\bm{\eta}}\left(\nw_{1,\ell,\bm{\eta}}^{(n)}: \ell \in [m] \setminus S_{k}\right)\Big| \leqslant 2\epsilon_{1} + S_{\beta} \epsilon_{3}.
\end{multline}
When $\mathcal{L}_{\bm{\chi}}$ is in $\mathcal{D}_{3}$, then for any $\mathcal{L}_{\bm{\eta}}$, we have, using \eqref{main_3_lem_5_eq_2} of Lemma~\ref{main_3_lem_5}:
\begin{multline}\label{main_3_part_1_upper_bound_4}
\left|F_{N,j,\bm{\chi}}\left(\bnw_{1,\bm{\chi}}^{(n)}\right) - F_{N,j,\bm{\eta}}\left(\bnw_{1,\bm{\eta}}^{(n)}\right)\right| \leqslant 2\epsilon_{1} + \mathcal{E}_{j,S_{j},\bm{\chi}} \max_{k \in [m]}\Big|G_{k,[m] \setminus S_{k},\bm{\chi}}\left(\nw_{1,\ell,\bm{\chi}}^{(n)}: \ell \in [m] \setminus S_{k}\right) \\- G_{k,[m] \setminus S_{k},\bm{\eta}}\left(\nw_{1,\ell,\bm{\eta}}^{(n)}: \ell \in [m] \setminus S_{k}\right)\Big| \leqslant 2\epsilon_{1} + \mathcal{E}_{j,S_{j},\bm{\chi}} \epsilon_{3}.
\end{multline}
Note that $S_{\alpha}$, $S_{\beta}$, $\mathcal{E}_{\bm{\chi}}$ and $\mathcal{E}_{j,S_{j},\bm{\chi}}$ are constants dependent only on $\bm{\chi}$. By the induction hypothesis, given $\epsilon_{2} > 0$, there exists $\epsilon_{1} > 0$ such that 
\begin{equation}\label{main_3_part_1_upper_bound_5}
d_{0}\left(\mathcal{L}_{\bm{\chi}}, \mathcal{L}_{\bm{\eta}}\right) \leqslant \epsilon_{1} \implies \max_{j \in [m]} \left|\nw_{1,j,\bm{\chi}}^{(n)} - \nw_{1,j,\bm{\eta}}^{(n)}\right| \leqslant \epsilon_{2}.
\end{equation}
Combining \eqref{main_3_part_1_upper_bound_1}, \eqref{main_3_part_1_upper_bound_2}, \eqref{epsilon_{3}_defn}, \eqref{main_3_part_1_upper_bound_3}, \eqref{main_3_part_1_upper_bound_4} and \eqref{main_3_part_1_upper_bound_5}, we get:
\begin{itemize}
\item For $\mathcal{L}_{\bm{\chi}} \in \mathcal{D}_{1} \cap \mathcal{D}_{2}$ and $d_{0}\left(\mathcal{L}_{\bm{\chi}}, \mathcal{L}_{\bm{\eta}}\right) \leqslant \epsilon_{1}$ with $\epsilon_{1}$ satisfying \eqref{dist_laws_main_3_part_1}, \eqref{dist_cond_main_3_part_1} and \eqref {main_3_part_1_upper_bound_5}, we have
\begin{equation}\label{case_1_main_3_part_1}
\left|F_{N,j,\bm{\chi}}\left(\bnw_{1,\bm{\chi}}^{(n)}\right) - F_{N,j,\bm{\eta}}\left(\bnw_{1,\bm{\eta}}^{(n)}\right)\right| \leqslant 2\epsilon_{1} + S_{\beta} \left(2\epsilon_{1} + S_{\alpha}\epsilon_{2}\right).
\end{equation}
\item For $\mathcal{L}_{\bm{\chi}} \in \mathcal{D}_{1} \cap \mathcal{D}_{3}$ and $d_{0}\left(\mathcal{L}_{\bm{\chi}}, \mathcal{L}_{\bm{\eta}}\right) \leqslant \epsilon_{1}$ with $\epsilon_{1}$ satisfying \eqref{dist_laws_main_3_part_1} and \eqref {main_3_part_1_upper_bound_5}, we have
\begin{equation}\label{case_1_main_3_part_2}
\left|F_{N,j,\bm{\chi}}\left(\bnw_{1,\bm{\chi}}^{(n)}\right) - F_{N,j,\bm{\eta}}\left(\bnw_{1,\bm{\eta}}^{(n)}\right)\right| \leqslant 2\epsilon_{1} + \mathcal{E}_{j,S_{j},\bm{\chi}} (2\epsilon_{1} + S_{\alpha}\epsilon_{2}).
\end{equation}
\item For $\mathcal{L}_{\bm{\chi}} \in \mathcal{D}_{4} \cap \mathcal{D}_{2}$ and $d_{0}\left(\mathcal{L}_{\bm{\chi}}, \mathcal{L}_{\bm{\eta}}\right) \leqslant \epsilon_{1}$ with $\epsilon_{1}$ satisfying \eqref{dist_cond_main_3_part_1} and \eqref {main_3_part_1_upper_bound_5}, we have
\begin{equation}\label{case_1_main_3_part_3}
\left|F_{N,j,\bm{\chi}}\left(\bnw_{1,\bm{\chi}}^{(n)}\right) - F_{N,j,\bm{\eta}}\left(\bnw_{1,\bm{\eta}}^{(n)}\right)\right| \leqslant 2\epsilon_{1} + S_{\beta} \left(2\epsilon_{1} + \mathcal{E}_{\bm{\chi}}\epsilon_{2}\right).
\end{equation}
\item For $\mathcal{L}_{\bm{\chi}} \in \mathcal{D}_{4} \cap \mathcal{D}_{3}$ and $d_{0}\left(\mathcal{L}_{\bm{\chi}}, \mathcal{L}_{\bm{\eta}}\right) \leqslant \epsilon_{1}$ with $\epsilon_{1}$ satisfying \eqref {main_3_part_1_upper_bound_5}, we have
\begin{equation}\label{case_1_main_3_part_4}
\left|F_{N,j,\bm{\chi}}\left(\bnw_{1,\bm{\chi}}^{(n)}\right) - F_{N,j,\bm{\eta}}\left(\bnw_{1,\bm{\eta}}^{(n)}\right)\right| \leqslant 2\epsilon_{1} + \mathcal{E}_{j,S_{j},\bm{\chi}} \left(2\epsilon_{1} + \mathcal{E}_{\bm{\chi}}\epsilon_{2}\right).
\end{equation}
\end{itemize}
From \eqref{main_3_part_1_upper_bound_5}, choosing $\epsilon_{1}$ arbitrarily small, we can make $\epsilon_{2}$ arbitrarily small, and thereby make the upper bound in each of \eqref{case_1_main_3_part_1}, \eqref{case_1_main_3_part_2},\eqref{case_1_main_3_part_3} and \eqref{case_1_main_3_part_4} arbitrarily small. This concludes the proof.

The proof of \ref{main_3_part_2} is almost entirely the same as that of \ref{main_3_part_1}. We only point out that, when $\mathcal{L}_{\bm{\chi}} \in \mathcal{C}_{1}$, we bound $\mw_{1,\ell,\bm{\chi}}^{(n)}$ using Lemma~\ref{main_3_lem_4}. When $\mathcal{L}_{\bm{\chi}} \in \mathcal{C}_{2}$, Lemma~\ref{main_3_lem_4} and the recursions in \S\ref{subsec:misere_recursions} yield
\begin{multline}
1 - G_{k,[m] \setminus S_{k},\bm{\chi}}\left(\mw_{1,\ell,\bm{\chi}}^{(n)}: \ell \in [m] \setminus S_{k}\right) + \beta_{k,\bm{\chi}} = 1 - \ml_{2,k}^{(n+1)}\\ \leqslant 1 - G_{k,[m] \setminus S_{k},\bm{\chi}}\left(\alpha_{\ell,\bm{\chi}}: \ell \in [m] \setminus S_{k}\right) + \beta_{k,\bm{\chi}} < 1.\nonumber
\end{multline}
We skip the proof of \ref{main_3_part_3} entirely as the argument follows the exact same lines as those of \ref{main_3_part_1} and \ref{main_3_part_2}.

When $\nd_{1,j,\bm{\chi}} = 0$, we have $\nw_{1,j,\bm{\chi}} = 1 - \nl_{1,j,\bm{\chi}}$. By \ref{main_3_part_1}, we know that $\nl_{1,j,\bm{\chi}}$ is lower semicontinuous, so that $1 - \nl_{1,j,\bm{\chi}}$ is upper semicontinuous in $\mathcal{L}_{\bm{\chi}}$. This makes $\nw_{1,j,\bm{\chi}}$, again by \ref{main_3_part_1}, both a lower and an upper semicontinuous function. Hence $\nw_{1,j,\bm{\chi}}$ is a continuous function of $\mathcal{L}_{\bm{\chi}}$. 
\end{proof}

\section{Proof of Theorem~\ref{thm:main_4}}\label{sec:main_4_proof}
We use a \emph{forcing strategy} somewhat similar to that of [\cite{holroyd_martin}, Proposition 13]. The root $\phi$ of $\mathcal{T}$ is the initial vertex, and Escaper plays the first round. Let $\mathcal{T}'$ be the subtree comprising paths $\mathcal{P} = (u_{0}, u_{1}, u_{2}, \ldots)$ that satisfy the following conditions (recall $f$ from Theorem~\ref{thm:main_4}): 
\begin{itemize}
\item $u_{0} = \phi$, i.e.\ the root belongs to each such path,
\item $\sigma(u_{2n+1}) \in [m] \setminus S_{\sigma(u_{2n})}$ and $\sigma(u_{2n+2}) \in S_{\sigma(u_{2n+1})}$ for every $n \in \mathbb{N}_{0}$,
\item each $u_{2n+1}$ has precisely one child (\emph{viz}.\ $u_{2n+2}$) whose colour is in $S_{\sigma(u_{2n+1})}$, and $\sigma(u_{2n+2}) = f\left(\sigma(u_{2n+1})\right)$ (note: $u_{2n+1}$ can have any number of children with colours in $[m] \setminus S_{\sigma(u_{2n+1})}$).
\end{itemize}  
If $\mathcal{T}'$ is infinite, and Escaper keeps the game confined to $\mathcal{T}'$, which she can, then each player is \emph{always} able to make a move, leading to a win for Escaper. We remove all the odd-leveled vertices from $\mathcal{T}'$, resulting in another rooted tree $\mathcal{T}''$, so that $\mathcal{T}'$ is infinite iff $\mathcal{T}''$ is infinite. We now seek a criterion that guarantees the survival of $\mathcal{T}''$ with positive probability.

Suppose $u_{2n} \in V\left(\mathcal{T}''\right)$ with $\sigma(u_{2n}) = i$ for some $i \in [m]$, and $v$ is a child of $u_{2n}$ in $\mathcal{T}$ with $\sigma(v) = k$ for some $k \in [m] \setminus S_{i}$. Call $v$ \emph{special} if, of all the children of $v$, there is precisely one, say $w$, whose colour belongs to the set $S_{k}$, and $\sigma(w) = f(k)$. The event that $v$ is special is independent of the subtrees $\mathcal{T}(v')$ of all other children $v'$ of $u_{2n}$, and the probability of this event is given by $\gamma_{k,f(k)}$. Thus, for each $k \in [m] \setminus S_{i}$, the number of special children $v$ of $u_{2n}$ with $\sigma(v) = k$ is given by $\bin\left(X_{u_{2n},k}, \gamma_{k,f(k)}\right)$. The total number of children of colour $j$ of $u_{2n}$ in $\mathcal{T}''$ is distributed as $\sum_{k \in [m] \setminus S_{i}: f(k) = j} \bin\left(X_{u_{2n},k}, \gamma_{k,j}\right)$, which yields
\begin{equation}
m''_{i,j} = \sum_{k \in [m] \setminus S_{i}: f(k) = j} \E\left[\bin\left(X_{u_{2n},k}, \gamma_{k,j}\right)\big|\sigma(u_{2n}) = i\right] = \sum_{k \in [m] \setminus S_{i}: f(k) = j} m_{i,k} \gamma_{k,j}.\nonumber
\end{equation}
By [\cite{athreya_ney}, Chapter V, \S3, Theorem 2], we conclude that if, for some choice of $f$, the largest eigenvalue of the mean matrix $M'' = \left(\left(m''_{i,j}\right)\right)_{i,j \in [m]}$, is strictly bigger than $1$, then $\mathcal{T}''$ survives with positive probability, thus implying that $\eew_{\sigma(\phi)} > 0$.

Assume $\eew_{j} > 0$ for all $j \in [m]$, and $\sigma(\phi) = i$. Since $\alpha_{i} < 1$, there exist non-negative integers $n_{k}$ for every $k \in S_{i}$, not all $0$ simultaneously, such that $\Prob\left[X_{\phi,k} = n_{k} \text{ for all } k \in S_{i}\big|\sigma(\phi) = i\right] > 0$. Name the children of $\phi$ of colour $k$ as $v_{k,1}, \ldots, v_{k,n_{k}}$. The probability that for every $k \in S_{i}$, every $v_{k,t}$ is in $\EEW_{k}$ for $1 \leqslant t \leqslant n_{k}$, equals $\Prob\left[X_{\phi,k} = n_{k} \text{ for all } k \in S_{i}\big|\sigma(\phi) = i\right] \prod_{k \in S_{i}}\left(\eew_{k}\right)^{n_{k}}$, which by our hypothesis is strictly positive. If $\phi$ is the initial vertex and Stopper plays the first round, Stopper is forced to move the token to some child $v_{k,t}$ of $\phi$ for some $k \in S_{i}$, thus allowing Escaper to win. This yields $\esl_{i} > 0$. A similar argument yields the converse.

\bibliography{Mult_GW_games}
\end{document}